\documentclass[12pt]{amsart}
\usepackage{amsmath,amssymb}
\usepackage{amsmath}
\usepackage{amsthm}
\usepackage{amsfonts}
\usepackage{graphicx}
\usepackage{float}
\usepackage{enumerate}
\usepackage{extarrows}
\usepackage{amssymb}
\usepackage{mathtools}
\usepackage{mathrsfs}
\usepackage{amsmath,amscd}
\usepackage{xy}
\xyoption{all}

\theoremstyle{plain}
\newtheorem{thm}{Theorem}[section]
\newtheorem{theorem}[thm]{Theorem}

\newtheorem{lemma}[thm]{Lemma}
\newtheorem{corollary}[thm]{Corollary}
\newtheorem{proposition}[thm]{Proposition}
\theoremstyle{definition}
\newtheorem{remark}[thm]{Remark}

\newtheorem{definition}[thm]{Definition}

\numberwithin{equation}{section}

\newcommand{\sB}{{\mathcal B}}

\newcommand{\sD}{{\mathcal D}}

\newcommand{\sH}{{\mathcal H}}

\newcommand{\sL}{{\mathcal L}}

\newcommand{\sU}{{\mathcal U}}

\newcommand{\sW}{{\mathcal W}}

\newcommand{\F}{{\mathbb F}}

\newcommand{\N}{{\mathbb N}}

\newcommand{\R}{{\mathbb R}}

\newcommand{\fg}{{\mathfrak g}}
\newcommand{\fgl}{{\mathfrak g}{\mathfrak l}}

\newcommand{\fh}{{\mathfrak h}}

\newcommand{\lcross}{{>\!\!\!\triangleleft}}

\title[Deformation quantization of symplectic vector fields]{Deformation quantization of symplectic vector fields}
\author{Haoyuan Gao}
\begin{document}

\begin{abstract}
In this paper, we study deformation quantization of symplectic vector fields \`a la Fedosov. We provide an explicit construction to quantize each symplectic vector field as a derivation of the deformed star algebra. This recovers a result of M\"uller-Bahns and Neumaier with an explicit formula of higher total degree terms. Moreover, we show that this quantization yields a non-abelian $2$-cocycle on the Lie algebra of symplectic vector fields with values in the deformed star algebra. For a Fedosov deformation quantization with vanishing first characteristic class, this non-abelian $2$-cocycle can provide a formal deformation of the crossed-product algebra generated by the action of symplectic vector fields. In this sense, we can quantize any Lie algebra action by symplectic vector fields.
\end{abstract}

\maketitle  \setcounter{tocdepth}{1}

\section{Introduction}

Deformation quantization was introduced by Bayen, Flato, Fronsdal, Lichnerowicz and Sternheimer in \cite{BFFLS}. For a Poisson manifold $(M, \{\cdot\,, \cdot\})$, a deformation quantization is an associative $\R[[\hbar]]$-bilinear product
\[\star : C^{\infty}(M)[[\hbar]] \times C^{\infty}(M)[[\hbar]] \to C^{\infty}(M)[[\hbar]], \,\,\,\, (f, g) \mapsto f \star g, \,\,\,\, \forall f, g \in C^{\infty}(M)[[\hbar]],\]
such that
\[f \star g = fg + B_1(f, g)\hbar + B_2(f, g)\hbar^2 + \cdots\]
for all $f, g \in C^{\infty}(M)$, where $B_q$ are smooth bi-differential operators, and
\[\frac{f \star g - g \star f}{\hbar}\Big|_{\hbar = 0} = \{f, g\},\]
i.e.,
\[B_1(f, g) - B_1(g, f) = \{f, g\}.\]
Sometimes we also require that the star product $\star$ preserves the unit, i.e.,
\[1 \star f = f \star 1 = f, \quad \forall f \in C^{\infty}(M)[[\hbar]].\]

The existence of deformation quantization of symplectic manifolds was first proved by De Wilde and Lecomte \cite{DL}. Later, Fedosov independently provided a geometric construction of deformation quantization of symplectic manifolds using symplectic connections \cite{F1,F2}. Moreover, given a symplectic manifold with a fixed symplectic connection, Fedosov provided a canonical way to quantize the symplectic manifold. The classification of deformation quantization up to equivalence of symplectic manifolds was obtained independently by Nest-Tsygan \cite{NT2}, Deligne \cite{D}, and Bertelson-Cahen-Gutt \cite{BCG}. We refer to \cite{G} for a survey. It was proved that the equivalence classes of deformation quantizations of a symplectic manifold $M$ are parameterized by the formal second de Rham cohomology $H^2_{\rm dR}(M)[[\hbar]]$. Moreover, each deformation quantization of $M$ is equivalent to a Fedosov deformation quantization \cite{D, BCG}. See Remark \ref{rmk-classification} for more details. Fedosov deformation quantizations belong to a special and interesting type of star product, called natural star product. We refer to \cite{GR} for a background on natural star products and quantum moment maps. The existence of deformation quantization of general Poisson manifolds was proved by Kontsevich as a corollary of his famous formality theorem \cite{K}.

Based on his construction of deformation quantization, Fedosov developed an algebraic version of the Atiyah-Singer index theorem \cite{F2}. Fedosov's and Nest-Tsygan's results are now collectively known as the algebraic index theorem \cite{NT1, NT2, F2}. A crucial part of the algebraic index theorem is the existence of a canonical trace on the deformed star algebra $(C^{\infty}(M)[[\hbar]], \star)$. Fedosov proved the existence of such a trace by considering deformation quantization of symplectic diffeomorphisms \cite{F2}. It is worth noting that Fedosov's deformation quantization of symplectic diffeomorphisms has applications in noncommutative geometry and quantum group theory \cite{TY,GY1}.

It is well known that the infinitesimal version of the group of symplectic diffeomorphisms is the Lie algebra of symplectic vector fields. A natural question is whether, given a symplectic vector field, there exists a derivation of the deformed star algebra with that classical limit. The answer is affirmative for general Poisson manifolds, as follows from Kontsevich's formality theorem \cite{S}. In fact, Kontsevich's formality map can induce a linear map from the space of Poisson vector fields to the space of derivations of the deformed star algebra. This linear map does not preserve the Lie brackets. Sharygin proved that the failure of this linear map to be a Lie algebra homomorphism is measured by inner derivations, and developed an obstruction theory for quantizing Lie algebra actions based on the Lichnerowicz-Poisson cohomology and the Lie algebra cohomology \cite{S}. For symplectic manifolds, there exist more explicit description and construction of deformation quantization of symplectic vector fields. For any deformation quantization $(C^{\infty}(M)[[\hbar]], \star)$ of a symplectic manifold $(M, \omega)$, Bertelson, Cahen and Gutt proved that the space of derivations of $(C^{\infty}(M)[[\hbar]], \star)$ is in bijection with the space of formal closed $1$-forms $Z^1_{\rm dR}(M)[[\hbar]]$, and the quotient space of derivations modulo quasi-inner derivations is in bijection with the formal first de Rham cohomology $H^1_{\rm dR}(M)[[\hbar]]$ \cite{BCG}. M\"uller-Bahns and Neumaier obtained the same result for Fedosov deformation quantizations with more explicit construction \cite{MN}. In particular, given a symplectic vector field $X \in \Gamma(TM)$, the closed $1$-form $- \iota_X \omega$ in M\"uller-Bahns and Neumaier's construction yields a derivation $\widetilde{X}$ with $\widetilde{X}|_{\hbar = 0} = X$.

In this paper, we provide an explicit and direct construction to quantize symplectic vector fields \`a la Fedosov. This recovers M\"uller-Bahns and Neumaier's construction of $\widetilde{X}$ with an explicit formula of higher total degree terms in the extended formal Weyl algebra. We then use this deformation quantization of symplectic vector fields to quantize Lie algebra actions. We point out that the quantization we construct does not preserve the Lie brackets if the chosen symplectic connection in Fedosov's deformation quantization is not invariant. However, the quantization always provides a non-abelian $2$-cocycle with values in the deformed star algebra. Therefore, we can use the extension theory of Hopf algebra actions \cite{Ma, Mo} to obtain a formal deformation of the crossed-product algebra of a Lie algebra action by symplectic vector fields. We remark that the extension theory of Hopf algebra actions has generalization to Hopf algebroids and applications in the study of Lie-Rinehart algebras \cite{BB, BKS}. In a forthcoming paper \cite{GY2}, we will generalize our construction to symplectic derivations in an algebraic setting, and use this method to study deformations of the Connes-Moscovici Hopf algebra $\sH_1$ of codimension $1$ foliations.

The organization of this paper is as follows. In Section 2, we review Fedosov's deformation quantization of symplectic manifolds. In Section 3, we provide a method to quantize a symplectic vector field by adding an inner derivation. The key ingredient here is a vanishing theorem for the quantum algebra-valued de Rham cohomology in positive degrees proved by Fedosov. In Section 4, we construct a deformation quantization of Lie algebra actions by using that of symplectic vector fields introduced in Section 3. The principal algebraic tools are cocycle crossed products and cleft extensions. In Appendix A, we compare our construction with that of M\"uller-Bahns and Neumaier. In Appendix B, we discuss cocycle phenomena of general deformation quantizations of symplectic vector fields.

\vspace{5mm}

{\em Acknowledgements:} The author would like to thank Yi-Jun Yao for many useful discussions and continued support. This work began when the author was a postdoctoral researcher at the Shanghai Center for Mathematical Sciences (SCMS). The author would like to thank SCMS for providing an excellent research environment.

The author would like to express sincere gratitude to the anonymous referee for the detailed and highly inspiring report, especially for pointing out the article \cite{MN}, which the author did not know before. Without the referee's recommendation, the author could not realize that (most of) the results of this paper are valid for general Fedosov deformation quantizations, not just for canonical ones (with zero characteristic class). This paper cannot be improved to the current version without the help of the referee.

All the ideas and proofs in this paper are 100\% organic human-made. But in the earlier version, Proposition \ref{prop-[D-L-X]-eta-X} and Theorem \ref{thm-cohomology-trivial} were stated for $\sW^+$-valued differential forms. This is an error because the images of $\sW^+ \otimes \wedge^{\bullet}$ under the operator $\delta$ and the Fedosov connection $D$ may not lie in $\sW^+ \otimes \wedge^{\bullet}$. The author realized this fact when reading a review provided by ChatGPT of another related paper. The extended algebra $\widetilde{\sW}^+$ is also suggested by ChapGPT in the same review. The author would like to thank ChatGPT for pointing out the above fact with the recommended extension $\widetilde{\sW}^+$. This corrects the author's wrong recognition for some years.

\section{Fedosov manifolds and their deformation quantization}

In this section, we review Fedosov's geometric construction of deformation quantization of symplectic manifolds. We consider real-valued smooth functions instead of complex-valued smooth functions as in Fedosov's original work.

Throughout this section, $(M, \omega)$ denotes a symplectic manifold of dimension $2n$. Let $f \in C^{\infty}(M)$ be a smooth function. We denote by $X_f$ the Hamiltonian vector field associated with $f$, i.e.,
\[\iota_{X_f}\omega = df,\]
or equivalently,
\[\omega(X_f, Y) = df(Y) = Yf, \quad \forall Y \in \Gamma(TM).\]
The canonical Poisson bracket on $M$ is given by
\[\{f, g\} = \omega(X_f, X_g), \quad \forall f, g \in C^{\infty}(M).\]
In local coordinates $(x^1, \dots, x^{2n})$, let $\omega_{ij} = \omega(\frac{\partial}{\partial x^i}, \frac{\partial}{\partial x^j})$. Then $\{x^i, x^j\} = - \omega^{ij}$, where $(\omega^{ij})$ is the inverse matrix of $(\omega_{ij})$.

\begin{definition}
    An affine connection $\nabla: \Gamma(TM) \times \Gamma(TM) \to \Gamma(TM)$ on $(M, \omega)$ is called a \textit{symplectic connection} if it is torsion-free and $\nabla\omega = 0$. A triple $(M, \omega, \nabla)$ with $\nabla$ a chosen symplectic connection is called a \textit{Fedosov manifold}.
\end{definition}

\begin{remark}
    Every symplectic manifold admits a symplectic connection, although it is not unique; cf. \cite{BCGRS}.
\end{remark}

In \cite{F1, F2}, Fedosov proved $(M, \omega)$ always admits a deformation quantization. Moreover, given a Fedosov manifold $(M, \omega, \nabla)$, Fedosov provided a canonical geometric construction of deformation quantization. We now briefly recall Fedosov's construction.

Fix $x \in M$. The formal Weyl algebra $\sW_x$ is defined to be the algebra of formal power series on the tangent space $T_xM$ endowed with the Moyal product associated to the symplectic vector space $(T_xM, \omega_x)$. More precisely, let $(x^1, \cdots, x^{2n})$ be local coordinates around $x$. The basis $(dx^1 |_x, \cdots, dx^{2n} |_x)$ gives linear coordinates $(y^1, \cdots, y^{2n})$ on $T_xM$. Then for $a, b \in \sW_x$, we have
\begin{equation*}
    a = \sum\limits_{k, l \geq 0}a_{i_1 \cdots i_k, l}y^{i_1}\cdots y^{i_k}\hbar^l, \quad b = \sum\limits_{k, l \geq 0}b_{i_1 \cdots i_k, l}y^{i_1}\cdots y^{i_k}\hbar^l,
\end{equation*}
where $i_1, \cdots, i_k \in \{1, 2, \cdots, 2n\}$ and $a_{i_1 \cdots i_k, l}, b_{i_1 \cdots i_k, l} \in \R$. The Moyal product of $a$ and $b$ is given by
\[a \circ b = \sum\limits_{q = 0}^{\infty}\frac{1}{q!}\left(-\frac{\hbar}{2}\right)^q\omega_x^{i_1 j_1} \cdots \omega_x^{i_q j_q}\frac{\partial^q a}{\partial y^{i_1} \cdots \partial y^{i_q}}\frac{\partial^q b}{\partial y^{j_1} \cdots \partial y^{j_q}}.\]
The vector bundle
\[\sW = \bigsqcup\limits_{x \in M}\sW_x\]
endowed with fiberwise Moyal product is called the \textit{formal Weyl algebra bundle}. Given a section $a : M \to \sW$, in a local coordinate chart $(U, x^1, \cdots, x^{2n})$, we have
\begin{equation} \label{section-W}
    a = \sum\limits_{k, l \geq 0}a_{i_1 \cdots i_k, l}y^{i_1}\cdots y^{i_k}\hbar^l,
\end{equation}
where $a_{i_1 \cdots i_k, l}$ are functions on $U$, and $(y^1, \cdots, y^{2n}) = (dx^1, \cdots, dx^{2n})$ is the associated local frame of $T^*M$ on $U$. The section $a$ is called \textit{smooth} if the coefficient functions $a_{i_1 \cdots i_k, l}$ are smooth in every local chart. We denote by $\Gamma(\sW)$ the space of smooth sections of $\sW$. For an open subset $U \subset M$, denote by $\Gamma(U, \sW)$ the space of smooth sections of $\sW$ on $U$. Given two smooth sections $a, b \in \Gamma(\sW)$, it is clear that their Moyal product $a \circ b$ is also smooth. For a homogeneous element $a = a_0y^{i_1} \cdots y^{i_k}\hbar^l$, its \textit{total degree} is defined to be $k + 2l$. Let
\[\sW_q = \left\{\sum\limits_{k + 2l \geq q}a_{i_1 \cdots i_k, l}y^{i_1}\cdots y^{i_k}\hbar^l \in \Gamma(\sW)\right\}.\]
We have a filtration of the algebra $\Gamma(\sW)$
\begin{equation} \label{filtration-W}
    \Gamma(\sW) = \sW_0 \supset \sW_1 \supset \sW_2 \supset \cdots
\end{equation}
that is compatible with the fiberwise Moyal product $\circ$. It is clear that this filtration is separated and complete, i.e.,
\[\bigcap_{q \in \N_0} \sW_q = \{ 0 \},\]
and
\[\sW \cong \lim_{\xleftarrow[q]{}}\sW / \sW_q.\]

\begin{remark} \label{rmk-center-W}
    Since $\omega$ is nondegenerate, the center of $(\Gamma(U, \sW), \circ)$ is $C^{\infty}(U)[[\hbar]]$ for any open subset $U \subset M$. Moreover, for $f \in C^{\infty}(U)[[\hbar]] \subset \Gamma(U, \sW)$ and $a \in \Gamma(U, \sW)$, their Moyal product $f \circ a$ coincides with the usual scalar product $fa$.
\end{remark}

We can also define $\sW$-valued smooth differential forms as smooth sections of $\sW \otimes \wedge^{\bullet}T^*M$. We denote by $\wedge^{\bullet} =\wedge^{\bullet}T^*M$ for simplicity. For $a \in \Gamma(\sW \otimes \wedge^{\bullet})$ and a local coordinate chart $(U, x^1, \cdots, x^{2n})$, we have
\[a|_U = \sum \hbar^q a_{i_1 \cdots i_k, j_1 \cdots j_l, q}y^{i_1}\cdots y^{i_k}dx^{j_1} \wedge \cdots \wedge dx^{j_l}.\]
For two $\sW$-valued differential forms $a$ and $b$, the fiberwise Moyal product $\circ$ and the wedge product $\wedge$ together provide a product of $a$ and $b$. For simplicity, we denote this product by $a \circ b$. Then we denote their graded commutator by $[a, b]$, i.e., if $a \in \Gamma(\sW \otimes \wedge^k)$, $b \in \Gamma(\sW \otimes \wedge^l)$, then
\[[a, b] = a \circ b - (-1)^{kl}b \circ a.\]
This definition is extended by linearity to non-homogeneous elements. Since the center of $\Gamma(\sW)$ is $C^{\infty}(M)[[\hbar]]$, we have, for $a \in \Gamma(\sW \otimes \wedge^{\bullet})$,
\[[a, b] = 0, \quad \forall b \in \Gamma(\sW \otimes \wedge^{\bullet}) \Longleftrightarrow a|_{y = 0} = a, \, \text{i.e.}, \, a \in \Gamma(\wedge^{\bullet}[[\hbar]]).\]
In other words, the (graded) center of $\Gamma(\sW \otimes \wedge^{\bullet})$ is the space of scalar forms. For any open subset $U \subset M$, the (graded) center of $\Gamma(U, \sW \otimes \wedge^{\bullet})$ is the space of scalar forms on $U$. The filtration \eqref{filtration-W} induces a separated and complete filtration of $\Gamma(\sW \otimes \wedge^{\bullet})$. We denote this filtration by
\[\Gamma(\sW \otimes \wedge^{\bullet}) = \sW_0 \otimes \wedge^{\bullet} \supset \sW_1 \otimes \wedge^{\bullet} \supset \sW_2 \otimes \wedge^{\bullet} \supset \cdots.\]

From now on, we write $\sW$ and $\sW \otimes \wedge^{\bullet}$ to denote $\Gamma(\sW)$ and $\Gamma(\sW \otimes \wedge^{\bullet})$, respectively. For $a \in \sW \otimes \wedge^{\bullet}$, denote by $a_0 \coloneqq a|_{y = 0}$, and $a_{00} \coloneqq a|_{y = 0, dx = 0}$. We also use the notation $\sigma(a) = a_0$. Fedosov introduced the following operators
\[\delta : \sW \otimes \wedge^{\bullet} \to \sW \otimes \wedge^{\bullet + 1}, \quad \delta a = dx^i \wedge \frac{\partial a}{\partial y^i},\]
and
\[\delta^* : \sW \otimes \wedge^{\bullet} \to \sW \otimes \wedge^{\bullet - 1}, \quad \delta^*a = y^i\iota_{\frac{\partial}{\partial x^i}}a.\]
Here $\iota_{\frac{\partial}{\partial x^i}}$ is the interior product. For a homogeneous section $a = y^{i_1} \cdots y^{i_k}dx^{j_1} \wedge \cdots \wedge dx^{j_l}$, let
\begin{equation*}
    \delta^{-1}(a) =
    \begin{cases}
        0 & \text{if}\,\,k + l = 0,\\
        \frac{1}{k + l}\delta^*(a) & \text{otherwise},
    \end{cases}
\end{equation*}
and extend $\delta^{-1}$ to general sections by linearity. Fedosov proved the following noncommutative version of the Hodge-de Rham decomposition.
\begin{lemma} \cite[Lemma 2.2]{F1} \label{lemma-Hodge-de Rham-decomposition}
    The above definitions of the operators $\delta$ and $\delta^*$ are independent of the choice of local coordinate charts, i.e., they are globally defined. Moreover, we have
    \[\delta \circ \delta = 0 = \delta^* \circ \delta^*,\]
    and
    \begin{equation} \label{Hodge-de Rham-decomposition-a}
        a = \delta\delta^{-1}a + \delta^{-1}\delta a + a_{00}, \quad \forall a \in \sW \otimes \wedge^{\bullet}.
    \end{equation}
\end{lemma}

\begin{remark}
    The operator $\delta$ is a graded derivation, i.e., for $a \in \sW \otimes \wedge^k$ and $b \in \sW \otimes \wedge^{\bullet}$, we have that
    \[\delta(a \circ b) = (\delta a) \circ b + (-1)^ka \circ (\delta b).\]
    However, $\delta^*$ is not a graded derivation. Moreover, we have
    \[\delta a = \left[\frac{1}{\hbar}\omega_{ij}y^i dx^j, a\right], \quad \forall a \in \sW \otimes \wedge^{\bullet}.\]
\end{remark}

The identity \eqref{Hodge-de Rham-decomposition-a} indicates that the operator $\delta^{-1}$ yields a chain homotopy of the complex $(\sW \otimes \wedge^{\bullet}, \delta)$ in positive degree terms. Therefore, the following corollary is immediate.

\begin{corollary} \label{cor-vanishing-delta-cohomology}
    The cohomology of the $\delta$-complex
    \[0 \to \sW \xrightarrow{\delta} \sW \otimes \wedge^1 \xrightarrow{\delta} \sW \otimes \wedge^2 \xrightarrow{\delta} \cdots \xrightarrow{\delta} \sW \otimes \wedge^{2n} \to 0\]
    vanishes in positive degrees.
\end{corollary}

The symplectic connection $\nabla$ also induces a graded derivation
\[\partial: \sW \otimes \wedge^{\bullet} \to \sW \otimes \wedge^{\bullet + 1}, \quad \partial a = dx^i \wedge \nabla_{\frac{\partial}{\partial x^i}}a.\]

Let $(x^1, \cdots, x^{2n})$ be a Darboux coordinate system on $U \subset M$, and $\Gamma^k_{ij}$ the Christoffel symbols of the connection $\nabla$ with respect to this coordinate system, i.e.,
\[\nabla_{\frac{\partial}{\partial x^i}}\frac{\partial}{\partial x^j} = \Gamma^k_{ij}\frac{\partial}{\partial x^k}.\]
Denote by
\[\Gamma_{ijk} = \omega_{il}\Gamma^l_{jk} = \omega\left(\frac{\partial}{\partial x^i}, \nabla_{\frac{\partial}{\partial x^j}}\frac{\partial}{\partial x^k}\right).\]
Let $R : \Gamma(TM) \times \Gamma(TM) \to {\rm End}(\Gamma(TM))$ be the curvature of $\nabla$. Let
\[R_{ijkl} = \omega\left(\frac{\partial}{\partial x^i}, R\left(\frac{\partial}{\partial x^k}, \frac{\partial}{\partial x^l}\right)\frac{\partial}{\partial x^j}\right)\]
be the coefficients of the symplectic curvature tensor. Fedosov introduced the following local $\sW$-valued differential forms
\begin{equation} \label{connection-curvature-form}
    \Gamma = \frac{1}{2}\Gamma_{ijk}y^iy^jdx^k, \quad R = \frac{1}{4}R_{ijkl}y^iy^jdx^k\wedge dx^l,
\end{equation}
and proved that
\begin{equation} \label{partial-Gamma-partial2-R}
    \partial a = d a - \left[\frac{1}{\hbar}\Gamma, a\right], \quad \partial^2 a= \partial(\partial a) = -\left[\frac{1}{\hbar}R, a\right]
\end{equation}
for all $a \in \Gamma(U, \sW \otimes \wedge^{\bullet})$, where $d$ is the de Rham differential with respect to the coordinates $(x^1, \cdots, x^{2n})$ of $U$. Note that the curvature form $R$ is globally defined because the center of $\Gamma(V, \sW \otimes \wedge^{\bullet})$ is $\Gamma(V, \wedge^{\bullet}[[\hbar]])$ for any open subset $V \subset M$, and $R|_{y = 0} = 0$ in any Darboux coordinate chart. However, the connection form $\Gamma \in \Gamma(U, \sW)$ is not globally defined because the de Rham differential $d$ depends on the choice of coordinate charts for $\sW$-valued differential forms.

\begin{remark} \label{rmk-Gamma-ijk-symmetric}
    As pointed out in \cite{F1}, $\Gamma_{ijk}$ are totally symmetric with respect to $(i, j, k)$ in Darboux coordinates because $\nabla$ is a symplectic connection. One can check that $R_{ijkl}$ are symmetric with respect to $(i, j)$ and anti-symmetric with respect to $(k, l)$. The identities \eqref{partial-Gamma-partial2-R} differ from those in \cite{F1} because we work with real-valued smooth functions.
\end{remark}

\begin{remark}
    A local frame $(e_1, \cdots, e_{2n})$ of $TM$ on an open subset $U \subset M$ is called a \textit{symplectic frame} if $\omega_{ij} = \omega(e_i, e_j)$ are constants for all $i, j$. It is easy to see that $\nabla \omega = 0$ if and only if
    \[\Gamma_{ijk} = \Gamma_{kji}\]
    under any local symplectic frame $(e_1, \cdots, e_{2n})$. Here $\nabla_{e_i}e_j = \Gamma^k_{ij}e_k$ and $\Gamma_{ijk} = \omega_{il}\Gamma^l_{jk}$. This observation is used in generalizations of Fedosov's construction to algebraic settings \cite{GY1, GY2}.
\end{remark}

Note that \eqref{partial-Gamma-partial2-R} indicates that $\partial^2 \neq 0$ if the connection $\nabla$ is not flat. Fedosov proved that adding an inner derivation to $\partial$ yields an abelian connection whose parallel sections are in one-to-one correspondence with the formal power series $C^{\infty}(M)[[\hbar]]$.

\begin{theorem} \cite[Theorem 3.2]{F1} \label{thm-Fedosov-1}
    Let $\Omega \in \hbar\wedge^2[[\hbar]]$ be a series of closed $2$-forms. There is a unique $r \in \sW_3 \otimes \wedge^1$ satisfying the equations
    \[\delta r = R + \partial r - \frac{1}{\hbar}r^2 + \Omega, \quad \delta^{-1}r = 0.\]
    Moreover, the derivation
    \[D : \sW \otimes \wedge^{\bullet} \to \sW \otimes \wedge^{\bullet + 1}, \quad Da = -\delta a + \partial a - \left[\frac{1}{\hbar}r, a\right]\]
    satisfies $D^2 = D \circ D = 0$.
\end{theorem}

Let $D$ be the operator in the preceding theorem, called the \textit{Fedosov connection}. Denote by $\sW_D = {\rm Ker}\,D \cap \sW$.
\begin{theorem} \cite[Theorem 3.3]{F1} \label{thm-Fedosov-2}
    For any $a_0 \in C^{\infty}(M)[[\hbar]]$, there is a unique section $a \in \sW_D$ such that $\sigma(a) = a_0$. Moreover, the section $a$ is given by the recursive formula
    \begin{equation} \label{recursive-formula-sigma-inverse-a}
        a = a_0 + \delta^{-1}(D + \delta)a = a_0 + \delta^{-1}\left(\partial a - \left[\frac{1}{\hbar}r, a\right]\right).
    \end{equation}
\end{theorem}
Since the operator $\left(\partial - \left[\frac{1}{\hbar}r, \cdot \,\right]\right)$ preserves the filtration and the operator $\delta^{-1}$ increases the filtration by $1$, \eqref{recursive-formula-sigma-inverse-a} admits a unique solution. As we shall use this recursive construction in our main results, we present a detailed proof of the existence and uniqueness of the solution of \eqref{recursive-formula-sigma-inverse-a}. In fact, the initial element $a_0$ can be chosen from $\sW \otimes \wedge^{\bullet}$.

\begin{lemma} \label{lemma-recursive-solution-a-a-q}
    For any $a_0 \in \sW \otimes \wedge^{\bullet}$, the equation \eqref{recursive-formula-sigma-inverse-a} admits a unique solution $a \in \sW \otimes \wedge^{\bullet}$.
\end{lemma}
\begin{proof}
    Define a sequence $\{a_q\}_{q \in \N_0} \subset \sW \otimes \wedge^{\bullet}$ by iterations of \eqref{recursive-formula-sigma-inverse-a}
    \[a_{q + 1} = a_0 + \delta^{-1}\left(\partial a_q - \left[\frac{1}{\hbar}r, a_q\right]\right), \quad q = 0, 1, 2, \cdots.\]
    Since the operator $\left(\partial - \left[\frac{1}{\hbar}r, \cdot \,\right]\right)$ preserves the filtration and the operator $\delta^{-1}$ increases the filtration by $1$, we have
    \begin{equation} \label{difference-a-1-a-0}
        a_1 - a_0 = \delta^{-1}\left(\partial a_0 - \left[\frac{1}{\hbar}r, a_0\right]\right) \in \sW_1 \otimes \wedge^{\bullet}.
    \end{equation}
    For $q \in \N$,
    \begin{equation} \label{induction-difference-a-q+1-a-q}
        \begin{aligned}
            a_{q + 1} - a_q = &\,\,(a_{q + 1} - a_0) - (a_q - a_0)\\
            = &\,\,\delta^{-1}\left(\partial a_q - \left[\frac{1}{\hbar}r, a_q\right]\right) - \delta^{-1}\left(\partial a_{q - 1} - \left[\frac{1}{\hbar}r, a_{q - 1}\right]\right)\\
            = &\,\,\delta^{-1}\left(\partial (a_q -a_{q - 1}) - \left[\frac{1}{\hbar}r, (a_q - a_{q - 1})\right]\right)
        \end{aligned}
    \end{equation}
    By induction, we may assume that $a_q - a_{q - 1} \in \sW_q \otimes \wedge^{\bullet}$, and \eqref{induction-difference-a-q+1-a-q} then indicates $a_{q + 1} - a_q \in \sW_{q + 1} \otimes \wedge^{\bullet}$. Hence, the summation
    \[a_0 + \sum_{q = 0}^{\infty}(a_{q + 1} - a_q)\]
    is well-defined as there are finitely many elements at each fixed total degree. Therefore, we obtain an element
    \[a = a_0 + \sum_{q = 0}^{\infty}(a_{q + 1} - a_q) \in \sW \otimes \wedge^{\bullet}.\]
    Substitution of \eqref{difference-a-1-a-0} and \eqref{induction-difference-a-q+1-a-q} indicates that $a$ satisfies
    \[a = a_0 + \delta^{-1}\left(\partial a - \left[\frac{1}{\hbar}r, a\right]\right).\]
    This ensures the existence of the solution of \eqref{recursive-formula-sigma-inverse-a}.
    
    Assume $a'$ is another solution of the equation \eqref{recursive-formula-sigma-inverse-a}. We then have
    \begin{equation}
        a -a ' = \delta^{-1}\left(\partial (a - a') - \left[\frac{1}{\hbar}r, (a - a')\right]\right).
    \end{equation}
    Since the operator $\left(\partial - \left[\frac{1}{\hbar}r, \cdot \,\right]\right)$ preserves the filtration and the operator $\delta^{-1}$ raises the filtration by $1$, we have
    \[a - a' \in \bigcap_{q \in \N} \sW_q \otimes \wedge^{\bullet} = \{ 0 \}.\]
    Therefore, the uniqueness of the solution of \eqref{recursive-formula-sigma-inverse-a} is proved.
\end{proof}

\begin{remark}
    In Fedosov's original paper \cite{F1}, the formal closed $2$-form $\Omega$ is taken to be zero. We refer to \cite{BCG} for a detailed proof of Theorem \ref{thm-Fedosov-1} and Theorem \ref{thm-Fedosov-2} with an arbitrary formal closed $2$-form $\Omega$.
\end{remark}

By Theorems \ref{thm-Fedosov-1} and \ref{thm-Fedosov-2}, we obtain an associative product on $C^{\infty}(M)[[\hbar]]$
\begin{equation} \label{star-product-Fedosov}
    a \star b \coloneqq \sigma(\sigma^{-1}(a) \circ \sigma^{-1}(b)), \quad \forall a, b \in C^{\infty}(M)[[\hbar]].
\end{equation}
In fact, the definition of $\partial$, $\delta^{-1}$, $r$ and the formula \eqref{recursive-formula-sigma-inverse-a} indicate that \eqref{star-product-Fedosov} is a deformation quantization of $(M, \omega)$. We refer to \cite{F1} for details. We call this the \textit{Fedosov deformation quantization} of the Fedosov manifold $(M, \omega, \nabla)$ with the \textit{characteristic form} $\frac{1}{\hbar}\Omega$. When $\Omega = 0$, we call \eqref{star-product-Fedosov} the \textit{canonical Fedosov deformation quantization} of $(M, \omega, \nabla)$.

\begin{remark} \label{rmk-classification}
    Two deformation quantizations $\ast$ and $\ast'$ of $M$ are said to be equivalent if there exists an $\R[[\hbar]]$-linear isomorphism
    \[\Psi : C^{\infty}(M)[[\hbar]] \to C^{\infty}(M)[[\hbar]]\]
    of the form
    \[\Psi = {\rm Id}_{C^{\infty}(M)} + \sum_{q = 1}^{\infty}\Psi_q \hbar^q,\]
    where each $\Psi_q$ is a smooth differential operator, such that
    \[f \ast' g = \Psi(\Psi^{-1}(f) \ast \Psi^{-1}(g)), \quad \forall f, g \in C^{\infty}(M)[[\hbar]].\]
    It was proved that the equivalence classes of deformation quantizations of $M$ are parametrized by $H^2_{\rm dR}(M)[[\hbar]]$ \cite{NT2, D, BCG}. Moreover, each deformation quantization is isomorphic to a Fedosov deformation quantization \cite{D, BCG}. Let $\Omega' \in \hbar \wedge^2[[\hbar]]$ be another series of closed $2$-forms. Let $\star'$ be the corresponding Fedosov deformation quantization. Then $\star$ and $\star'$ are equivalent if and only if $\Omega$ and $\Omega'$ are cohomologous \cite{F2, BCG}.
\end{remark}

For use in the following sections, we recall Fedosov's extension $\sW^{+}$ as follows:
\begin{itemize}
    \item[(i)] Elements $a \in \sW^{+}$ are given by the series \eqref{section-W}, but the powers of $\hbar$ can be both positive and negative.
    \item[(ii)] The total degree $k + 2l$ of any term of the series is nonnegative.
    \item[(iii)] There are finitely many terms of each fixed nonnegative total degree.
\end{itemize}

It is clear that $\sW^+$ is also an algebra with respect to the fiberwise Moyal product. We also have the algebra $\sW^+ \otimes \wedge^{\bullet}$ of $\sW^+$-valued smooth differential forms. The operators $\delta^{-1}$ and $\partial$ also act on $\sW^+ \otimes \wedge^{\bullet}$. Since $D = - \delta + \partial - \left[\frac{1}{\hbar}r, \cdot\,\right]$, and $r \in \sW_3 \otimes \wedge^1$, we have $\frac{1}{\hbar}r \in \sW^+ \otimes \wedge^1$, thus the operator $D +  \delta = \partial - \left[\frac{1}{\hbar}r, \cdot \,\right]$ also acts on $\sW^+ \otimes \wedge^{\bullet}$. We also have a complete filtration
\[\sW^+ \otimes \wedge^{\bullet} = \sW^+_0 \otimes \wedge^{\bullet} \supset \sW^+_1 \otimes \wedge^{\bullet} \supset \sW^+_2 \otimes \wedge^{\bullet} \supset \cdots\]
with respect to the total degree. Thus, Lemma \ref{lemma-recursive-solution-a-a-q} is also valid for $\sW^+ \otimes \wedge^{\bullet}$. For later use, we restate the lemma for $\sW^+ \otimes \wedge^{\bullet}$.
\begin{lemma} \label{lemma-recursive-solution-a-a-q-W-+}
    For any $a_0 \in \sW^+ \otimes \wedge^{\bullet}$, the equation \eqref{recursive-formula-sigma-inverse-a} admits a unique solution $a \in \sW^+ \otimes \wedge^{\bullet}$.
\end{lemma}
The image of $\sW^+ \otimes \wedge^{\bullet}$ under the operator $\delta$ may not lie in $\sW^+ \otimes \wedge^{\bullet}$ as $\delta$ decreases total degrees by $1$. Thus, the image of $\sW^+ \otimes \wedge^{\bullet}$ under the Fedosov connection $D$ may also have negative total degree terms. As we need to consider the action of $D$ on elements of $\sW^+ \otimes \wedge^{\bullet}$ later, we introduce the following extension of $\sW^+$:
\[\widetilde{\sW}^+ \coloneqq \bigcup_{N \in \N}\frac{1}{\hbar^N}\sW^+.\]
It is clear that $\widetilde{\sW}^+$ is also an algebra with respect to the fiberwise Moyal product, and $\sW^+$ is a subalgebra of $\widetilde{\sW}^+$. The operators $\delta$, $\delta^{-1}$, $\partial$ and $D$ are naturally extended to $\widetilde{\sW}^+$-valued differential forms $\widetilde{\sW}^+ \otimes \wedge^{\bullet}$. Moreover, Fedosov proved the following lemma.
\begin{lemma} \cite[Lemma 4.2]{F1} \label{lemma-D-W-+}
    Let $a \in \sW^+$ and $Da = 0$. Then $a$ does not contain negative powers of $\hbar$, i.e., $a \in \sW_D \subset \sW \subset \sW^+$.
\end{lemma}

\begin{remark}
    Fedosov introduced the algebra $\sW^+$ to provide a method to quantize symplectic diffeomorphisms \cite{F1}, based on which he constructed a canonical trace on the deformed star algebra \eqref{star-product-Fedosov}. Fedosov's method is also used in studies of deformation theory in noncommutative geometry and quantum groups \cite{TY, GY1}.
\end{remark}

\section{Deformation quantization of symplectic vector fields}

In this section, we provide a method to quantize symplectic vector fields \`a la Fedosov. Throughout this section, $(M, \omega, \nabla)$ denotes a Fedosov manifold. Let $\Omega \in \hbar\wedge^2[[\hbar]]$ be a fixed series of closed $2$-forms. We use the notations $\delta$, $\delta^{-1}$, $\partial$, $D$, $\sW_D$$\cdots$ as in Section 2.

We first recall some standard formulas concerning the Lie derivative and the de Rham differential.

\begin{lemma} \label{lemma-L-iota}
    Let $X, Y \in \Gamma(TM)$. Then we have
    \begin{itemize}
        \item [(i)] $\sL_X = [\iota_X, d] = \iota_X \circ d + d \circ \iota_X : \Gamma(\wedge^{\bullet}T^*M) \to \Gamma(\wedge^{\bullet}T^*M)$;
        \item [(ii)] $[\sL_X, \iota_Y] = \sL_X \circ \iota_Y - \iota_Y \circ \sL_X = \iota_{[X, Y]} : \Gamma(T^{\bullet}T^*M) \to \Gamma(T^{\bullet - 1}T^*M)$, where
        \[T^{\bullet}T^*M = \bigoplus\limits_{q = 0}^{\infty}(T^*M)^{\otimes q} = \bigoplus\limits_{q = 0}^{\infty}\underbrace{T^*M \otimes \cdots \otimes T^*M}_{q \,\, {\rm times}}\]
        is the tensor bundle of the cotangent bundle $T^*M$.
    \end{itemize}
    Here
    \[\iota_X : \Gamma(T^{\bullet}T^*M) \to \Gamma(T^{\bullet - 1}T^*M)\]
    is the interior product defined by
    \[(\iota_X\tau)(Y_1, \cdots, Y_{k - 1}) = \tau(X, Y_1, \cdots, Y_{k - 1}), \quad \forall \tau \in \Gamma((T^*M)^{\otimes k}),\]
    for all $Y_1, \cdots, Y_{k - 1} \in \Gamma(TM)$.
\end{lemma}

The first formula in Lemma \ref{lemma-L-iota} is the Cartan homotopy formula. The second formula can be checked by the global formula for the Lie derivative. We refer to \cite{LPV} for a proof. The following corollary is immediate from (i) of Lemma \ref{lemma-L-iota}.

\begin{corollary} \label{corollary-d-L-X-L-X-d}
    Let $X \in \Gamma(TM)$. Then we have
    \[d \circ \sL_X = \sL_X \circ d : \Gamma(\wedge^{\bullet}T^*M) \to \Gamma(\wedge^{\bullet + 1}T^*M).\]
\end{corollary}

\begin{definition}
    A smooth vector field $X \in \Gamma(TM)$ on $(M, \omega)$ is called a \textit{symplectic vector field} if
    \[\sL_X \omega = 0.\]
    Equivalently, $X$ is symplectic if the flow of $X$ consists of local symplectic diffeomorphisms.
\end{definition}

Note that if $(C^{\infty}(N)[[\hbar]], \star)$ is a deformation quantization of a Poisson manifold $(N, \{\cdot\,,\cdot\})$, and
\begin{equation} \label{derivation-star}
    \widetilde{X} = X_0 + X_1\hbar + X_2 \hbar^2 + \cdots
\end{equation}
with $X_q \in {\rm Hom}_{\R}(C^{\infty}(N), C^{\infty}(N))$ for each $q \in \N_0$ is a derivation of $(C^{\infty}(N)[[\hbar]], \star)$, then $X_0$ must be a smooth vector field satisfying
\[\sL_{X_0}\{\cdot \,,\cdot\} = 0.\]
When $N$ is symplectic and the Poisson bracket is induced by its symplectic form, this condition is equivalent to $X_0$ being a symplectic vector field. In the following, we show that if the star product $\star$ is the Fedosov deformation quantization \eqref{star-product-Fedosov}, then given a symplectic vector field $X_0$, we can extend $X_0$ to a derivation of $(C^{\infty}(M)[[\hbar]], \star)$ of the form \eqref{derivation-star}.

\begin{lemma} \label{lemma-L-X}
    Let $X \in \Gamma(TM)$ be a symplectic vector field on $M$. The Lie derivative along $X$ gives a derivation
    \[\sL_X : \sW \otimes \wedge^{\bullet} \to \sW \otimes \wedge^{\bullet}\]
    that preserves both the symmetric degree and the exterior degree.
\end{lemma}
\begin{proof}
    Let $\theta : \sD \to M$ be the (maximal) flow of $X$, where the open subset $\sD \subset \R \times M$ is the domain of the flow. Denote by
    \[\sD_t = \{p \in M \, | \, (t, p) \in \sD\},\]
    and
    \[\theta_t : \sD_t \to \sD_{-t}, \quad \theta_t(p) = \theta(t, p), \quad \forall p \in \sD_t\]
    for each $t \in \R$. Since $X$ is a symplectic vector field, we have
    \[\theta_t^*(\omega|_{\sD_{-t}}) = \omega|_{\sD_t}, \quad \forall t \in \R.\]
    Therefore, we have algebra isomorphisms
    \[\theta_t^* : \Gamma(\sD_{-t}, \sW \otimes \wedge^{\bullet}) \to \Gamma(\sD_t, \sW \otimes \wedge^{\bullet})\]
    for all $t \in \R$. Then for $a \in \sW \otimes \wedge^{\bullet}$ and $p \in M$, by definition we have
    \begin{equation*}
        \begin{aligned}
            (\sL_X a)_p = &\,\,\lim\limits_{t \to 0}\frac{(d\theta_t)_p^*(a_{\theta_t(p)}) - a_p}{t}\\
            = &\,\,\frac{d}{dt}\Big|_{t = 0}(\theta_t^* a)_p.
        \end{aligned}
    \end{equation*}
    For $a, b \in \sW \otimes \wedge^{\bullet}$, $p \in M$, we have
    \begin{equation*}
        \begin{aligned}
            (\sL_X (a \circ b))_p = &\,\,\frac{d}{dt}\Big|_{t = 0}(\theta_t^* (a \circ b))_p\\
            = &\,\,\frac{d}{dt}\Big|_{t = 0}((\theta_t^* a) \circ (\theta_t^*b) )_p\\
            = &\,\,\frac{d}{dt}\Big|_{t = 0}(\theta_t^* a)_p \circ b_p + a_p \circ \frac{d}{dt}\Big|_{t = 0}(\theta_t^* b)_p\\
            = &\,\,(\sL_Xa)_p \circ b_p + a_p \circ (\sL_X b)_p.
        \end{aligned}
    \end{equation*}
    Thus,
    \begin{equation} \label{L-X-a-circ-b}
        \sL_X(a \circ b) = (\sL_X a) \circ b + a \circ (\sL_X b).
    \end{equation}
    Therefore, $\sL_X$ is a derivation of the algebra $\sW \otimes \wedge^{\bullet}$. By definition, $\sL_X$ preserves both the symmetric degree and the exterior degree.
\end{proof}

\begin{remark} \label{rmk-L-X}
    As $\sL_X$ preserves the symmetric degree, it also gives derivations
    \[\sL_X : \sW^+ \otimes \wedge^{\bullet} \to \sW^+ \otimes \wedge^{\bullet}\]
    and
    \[\sL_X : \widetilde{\sW}^+ \otimes \wedge^{\bullet} \to \widetilde{\sW}^+ \otimes \wedge^{\bullet}.\]
    The above lemma can also be proved using the global formula for the Lie derivative
    \[(\sL_X \omega)(Y, Z) = X\omega(Y, Z) - \omega(\sL_X Y, Z) - \omega(Y, \sL_X Z), \quad \forall Y, Z \in \Gamma(TM).\]
    Hence $\sL_X \omega = 0$ if and only if in any Darboux coordinates $(x^1, \cdots, x^{2n})$, we have
    \begin{equation} \label{L-X-omege-0-local}
        \frac{\partial X^k}{\partial x^i}\omega_{kj} = \frac{\partial X^k}{\partial x^j}\omega_{ki},
    \end{equation}
    for all $i$ and $j$, where $X = X^k \frac{\partial}{\partial x^k}$. Then by \eqref{L-X-omege-0-local} and the fact (see the proof of Proposition \ref{prop-[partial-L-X]} later) that
    \[\sL_X dx^i = \frac{\partial X^i}{\partial x^j}dx^j, \quad \sL_X y^i = \frac{\partial X^i}{\partial x^j}y^j, \quad \forall i,\]
    one can check \eqref{L-X-a-circ-b} in local Darboux coordinates.
\end{remark}

If $X, Y \in \Gamma(TM)$ are both symplectic vector fields, it is clear that $[X, Y]$ is also symplectic. Denote by $\Gamma^{\omega}(TM)$ the Lie algebra of symplectic vector fields. For $X \in \Gamma^{\omega}(TM)$, define the Lie derivative of $\nabla$ by
\[(\sL_X \nabla)(Y, Z) = [X, \nabla_Y Z] - \nabla_{[X, Y]}Z - \nabla_Y[X, Z], \quad \forall Y, Z \in \Gamma(TM).\]
Note that $\sL_X \nabla \in \Gamma((T^*M)^{\otimes 2} \otimes TM)$ is a tensor field rather than a connection; hence
\begin{equation} \label{def-nable-X}
    \nabla^X \coloneqq \nabla + \sL_X \nabla
\end{equation}
is a connection.
\begin{lemma}
    Let $X \in \Gamma^{\omega}(TM)$. The connection $\nabla^X$ defined in \eqref{def-nable-X} is a symplectic connection.
\end{lemma}
\begin{proof}
    For $Y, Z \in \Gamma(TM)$, we have
    \begin{equation*}
        \begin{aligned}
            &\,\,(\sL_X \nabla)(Y, Z) - (\sL_X \nabla)(Z, Y)\\
            = &\,\,[X, \nabla_Y Z] - \nabla_{[X, Y]}Z - \nabla_Y [X, Z] - [X, \nabla_Z Y] + \nabla_{[X, Z]} Y + \nabla_Z [X, Y]\\
            = &\,\,[X, \nabla_Y Z -\nabla_Z Y] - (\nabla_{[X, Y]}Z - \nabla_Z [X, Y]) - (\nabla_Y [X, Z] - \nabla_{[X, Z]} Y).
        \end{aligned}
    \end{equation*}
    Since $\nabla$ is torsion-free, we have
    \begin{equation*}
        \begin{aligned}
            &\,\,(\sL_X \nabla)(Y, Z) - (\sL_X \nabla)(Z, Y)\\
            = &\,\,[X, \nabla_Y Z -\nabla_Z Y] - (\nabla_{[X, Y]}Z - \nabla_Z [X, Y]) - (\nabla_Y [X, Z] - \nabla_{[X, Z]} Y)\\
            = &\,\,[X, [Y, Z]] - [[X, Y], Z] - [Y, [X, Z]]\\
            = &\,\,0.
        \end{aligned}
    \end{equation*}
    We then have
    \begin{equation*}
        \begin{aligned}
            \nabla^X_Y Z - \nabla^X_Z Y = &\,\,\nabla_Y Z - \nabla_Z Y + (\sL_X \nabla)(Y, Z) - (\sL_X \nabla)(Z, Y)\\
            = &\,\,\nabla_Y Z - \nabla_Z Y\\
            = &\,\,[Y, Z].
        \end{aligned}
    \end{equation*}
    Hence, $\nabla^X$ is torsion-free.

    It remains to show
    \[\nabla^X \omega = 0,\]
    or equivalently,
    \[(\nabla^X_Y \omega)(Z, W) = 0, \quad \forall Y, Z, W \in \Gamma(TM).\]
    Let $Y, Z, W \in \Gamma(TM)$ be arbitrary. By definition,
    \begin{equation*}
        \begin{aligned}
            (\nabla^X_Y \omega)(Z, W) = &\,\,Y \omega(Z, W) - \omega(\nabla^X_Y Z, W) - \omega(Z, \nabla^X_Y W)\\
            = &\,\,Y \omega(Z, W) - \omega(\nabla_Y Z + (\sL_X \nabla)(Y, Z), W) - \omega(Z, \nabla_Y W + (\sL_X \nabla)(Y, W))\\
            = &\,\,(\nabla_Y \omega)(Z, W) - \omega((\sL_X \nabla)(Y, Z), W) - \omega(Z, (\sL_X \nabla)(Y ,W)).
        \end{aligned}
    \end{equation*}
    Since $\nabla$ is a symplectic connection, we have
    \[(\nabla_Y \omega)(Z, W) = 0.\]
    Thus,
    \begin{equation} \label{nabla-X-Y-omega-Z-W}
        \begin{aligned}
            &\,\,(\nabla^X_Y \omega)(Z, W)\\
            = &\,- \omega((\sL_X \nabla)(Y, Z), W) - \omega(Z, (\sL_X \nabla)(Y ,W))\\
            = &\,-\omega([X, \nabla_Y Z] - \nabla_{[X, Y]} Z - \nabla_Y [X, Z], W)\\
            &\,- \omega(Z, [X, \nabla_Y W] - \nabla_{[X, Y]} W - \nabla_Y [X, W])\\
            = &\,\,\omega(\nabla_{[X, Y]} Z, W) + \omega(Z, \nabla_{[X, Y]} W)\\
            &\,- \omega([X, \nabla_Y Z], W) + \omega(\nabla_Y [X, Z], W) - \omega(Z, [X, \nabla_Y W]) + \omega(Z, \nabla_Y[X, W]).
        \end{aligned}
    \end{equation}
    Since $\nabla\omega = 0$, and $\sL_X \omega = 0$, we have
    \begin{equation} \label{eq-XY-omega-Z-W}
        \omega(\nabla_{[X, Y]} Z, W) + \omega(Z, \nabla_{[X, Y]} W) = [X, Y]\omega(Z, W),
    \end{equation}
    and
    \begin{equation} \label{eq-nabla-X-omega-vanish-1}
        \begin{aligned}
            \omega([X, \nabla_Y Z], W) = &\,\,X \omega(\nabla_Y Z, W) - \omega(\nabla_Y Z, [X, W])\\
            = &\,\,XY \omega(Z, W) - X \omega(Z, \nabla_Y W) - \omega(\nabla_Y Z, [X, W]),
        \end{aligned}
    \end{equation}
    and
    \begin{equation} \label{eq-nabla-X-omega-vanish-2}
        \omega(\nabla_Y [X, Z], W) = Y \omega([X, Z], W) - \omega([X, Z], \nabla_Y W),
    \end{equation}
    and
    \begin{equation} \label{eq-nabla-X-omega-vanish-3}
        \omega(Z, [X, \nabla_Y W]) = X \omega(Z, \nabla_Y W) - \omega{([X, Z], \nabla_Y W)},
    \end{equation}
    and
    \begin{equation} \label{eq-nabla-X-omega-vanish-4}
        \begin{aligned}
            \omega(Z, \nabla_Y [X, W]) = &\,\,Y \omega(Z, [X, W]) - \omega(\nabla_Y Z, [X, W])\\
            = &\,\,YX \omega(Z, W) - Y\omega([X, Z], W) - \omega(\nabla_Y Z, [X, W]).
        \end{aligned}
    \end{equation}
    Substitution of \eqref{eq-nabla-X-omega-vanish-1}-\eqref{eq-nabla-X-omega-vanish-4} yields
    \begin{equation*}
        \begin{aligned}
            &\,- \omega([X, \nabla_Y Z], W) + \omega(\nabla_Y [X, Z], W) - \omega(Z, [X, \nabla_Y W]) + \omega(Z, \nabla_Y[X, W])\\
            = &\,- (XY - YX)\omega(Z, W)\\
            = &\,- [X, Y]\omega(Z, W).
        \end{aligned}
    \end{equation*}
    Then, by \eqref{nabla-X-Y-omega-Z-W} and \eqref{eq-XY-omega-Z-W},
    \[(\nabla^X_Y \omega)(Z, W) = \omega(\nabla_{[X, Y]} Z, W) + \omega(Z, \nabla_{[X, Y]} W) - [X, Y]\omega(Z, W) = 0,\]
    and the proof is complete.
\end{proof}
Let $X$ be a symplectic vector field. The preceding lemma indicates that the connection $\nabla^X$ defined in \eqref{def-nable-X} is also symplectic. Let $\partial^X$ be the graded derivation induced by $\nabla^X$, i.e.,
\begin{equation} \label{def-partial-X}
    \partial^X : \sW \otimes \wedge^{\bullet} \to \sW \otimes \wedge^{\bullet + 1}, \quad \partial^X a = dx^i \wedge \nabla^X_{\frac{\partial}{\partial x^i}}a.
\end{equation}

Let $(U, x^1, \cdots, x^{2n})$ be a Darboux coordinate chart. Then we have the connection form $\Gamma \in \Gamma(U, \sW_2 \otimes \wedge^1)$ of the connection $\nabla$ and the connection form $\Gamma^X \in \Gamma(U, \sW_2 \otimes \wedge^1)$ of the connection $\nabla^X$ as defined in \eqref{connection-curvature-form}. Then by \eqref{partial-Gamma-partial2-R}, we have that
\[(\partial - \partial^X)a = \left[\frac{1}{\hbar}(\Gamma^X - \Gamma), a\right], \quad \forall a \in \Gamma(U, \sW \otimes \wedge^{\bullet}).\]
Note that the operators $\partial$ and $\partial^X$ are globally defined, and the center of $\Gamma(V, \sW \otimes \wedge^{\bullet})$ is $\Gamma(V, \wedge^{\bullet}[[\hbar]])$ for any open subset $V \subset M$. Then since ${\Gamma^X}|_{y = 0} = \Gamma |_{y = 0} = 0$ in any local Darboux coordinate chart, we have that the difference $\Gamma^X - \Gamma$ defines a global $\sW$-valued $1$-form, i.e.,
\[\Gamma^X - \Gamma \in \sW_2 \otimes \wedge^1.\]
Set
\begin{equation} \label{def-eta-X}
    \eta_X = \Gamma^X - \Gamma + \sL_X r + \iota_X \Omega \in \sW_2 \otimes \wedge^1,
\end{equation}
where $r \in \sW_3 \otimes \wedge^1$ is given by Theorem \ref{thm-Fedosov-1}. Consider the recursively defined element $u_X \in \sW^+_1$:
\begin{equation} \label{def-u-X}
    u_X = \delta^{-1}\left((D + \delta)u_X + \frac{1}{\hbar}\eta_X\right).
\end{equation}
By Lemma \ref{lemma-recursive-solution-a-a-q-W-+}, $u_X$ is well-defined. In the following, we show that the derivation
\[\sL_X + [u_X, \cdot\,] : \sW^+ \to \sW^+\]
maps the subalgebra $\sW_D$ to itself.

We need the following lemma.

\begin{lemma} \label{lemma-nablaX-minus-nabla}
    Let $X \in \Gamma^{\omega}(TM)$ be a symplectic vector field. Let $Y \in \Gamma(TM)$ be a smooth vector field. Then the covariant derivatives
    \[\nabla^X_Y : \Gamma(T^{\bullet}T^*M) \to \Gamma(T^{\bullet}T^*M)\]
    and
    \[\nabla_Y : \Gamma(T^{\bullet}T^*M) \to \Gamma(T^{\bullet}T^*M)\]
    satisfy
    \[\nabla^X _Y = \nabla_Y + [\sL_X, \nabla_Y] - \nabla_{[X, Y]}.\]
\end{lemma}

\begin{proof}
    First note that
    \[([\sL_X, \nabla_Y] - \nabla_{[X, Y]})f =X(Y f) - Y(X f) - [X, Y]f = 0, \quad \forall f \in C^{\infty}(M).\]
    Hence
    \[\nabla^X _Y f = Y f = \nabla_Y f = (\nabla_Y + [\sL_X, \nabla_Y] - \nabla_{[X, Y]})f, \quad \forall f \in C^{\infty}(M).\]
    Then since $\nabla^X _Y$, $\nabla_Y$, $\sL_X$ and $\nabla_{[X, Y]}$ are derivations, it suffices to show that
    \[\nabla^X _Y \alpha = \nabla_Y \alpha + [\sL_X, \nabla_Y]\alpha - \nabla_{[X, Y]}\alpha, \quad \forall \alpha \in \wedge^1 = \Gamma(T^* M).\]
    Let $\alpha \in \wedge^1$ and $Z \in \Gamma(TM)$ be arbitrary. By definition, we have
    \begin{equation} \label{lemma-nablaX-minus-nabla-eq1}
        \begin{aligned}
            (\nabla^X_Y \alpha)(Z) = &\,\,Y(\alpha(Z)) - \alpha(\nabla^X_Y Z)\\
            = &\,\,Y(\alpha(Z)) - \alpha(\nabla_Y Z + [X, \nabla_Y Z] - \nabla_{[X, Y]}Z - \nabla_Y[X, Z]),
        \end{aligned}
    \end{equation}
    and
    \begin{equation} \label{lemma-nablaX-minus-nabla-eq2}
        (\nabla_Y \alpha)(Z) = Y(\alpha(Z)) - \alpha(\nabla_Y Z),
    \end{equation}
    and
    \begin{equation} \label{lemma-nablaX-minus-nabla-eq3}
        \begin{aligned}
            ([\sL_X, \nabla_Y]\alpha)(Z) = &\,\,(\sL_X(\nabla_Y \alpha))(Z) - (\nabla_Y(\sL_X \alpha))(Z)\\
            = &\,\,X((\nabla_Y \alpha)(Z)) - (\nabla_Y \alpha)([X, Z]) - Y((\sL_X \alpha)(Z)) + (\sL_X \alpha)(\nabla_Y Z)\\
            = &\,\,XY(\alpha(Z)) - X(\alpha(\nabla_Y Z)) - Y(\alpha([X, Z])) + \alpha(\nabla_Y [X, Z])\\
            &\,- YX(\alpha(Z)) + Y(\alpha([X, Z])) + X(\alpha(\nabla_Y Z)) - \alpha([X, \nabla_Y Z])\\
            = &\,\,[X, Y](\alpha(Z)) + \alpha(\nabla_Y[X, Z]) - \alpha([X, \nabla_Y Z]),
        \end{aligned}
    \end{equation}
    and
    \begin{equation} \label{lemma-nablaX-minus-nabla-eq4}
        (\nabla_{[X, Y]}\alpha)(Z) = [X, Y](\alpha(Z)) - \alpha(\nabla_{[X, Y]}Z).
    \end{equation}
    Combining \eqref{lemma-nablaX-minus-nabla-eq1}-\eqref{lemma-nablaX-minus-nabla-eq4}, we have
    \[(\nabla^X _Y \alpha)(Z) = (\nabla_Y \alpha)(Z) + ([\sL_X, \nabla_Y]\alpha)(Z) - (\nabla_{[X, Y]}\alpha)(Z).\qedhere\]
\end{proof}

Then we have the following proposition concerning the commutator of the Lie derivative $\sL_X$ and the derivation $\partial$.

\begin{proposition} \label{prop-[partial-L-X]}
    For $X \in \Gamma^{\omega}(TM)$ and $a \in \sW \otimes \wedge^{\bullet}$,
    \[[\partial, \sL_X]a = \partial(\sL_X a) - \sL_X(\partial a) = (\partial - \partial^X)a = \left[\frac{1}{\hbar}(\Gamma^X - \Gamma), a\right].\]
\end{proposition}

\begin{proof}
    In local Darboux coordinates $(U, x^1, \cdots, x^{2n})$, write $X = X^i \frac{\partial}{\partial x^i}$ with $X^i \in C^{\infty}(U)$. Then we have
    \begin{equation} \label{bracket-X-partial-x-i}
        \left[X, \frac{\partial}{\partial x^i}\right] = - \frac{\partial X^j}{\partial x^i}\frac{\partial}{\partial x^j}, \quad \forall i.
    \end{equation}
    By Corollary \ref{corollary-d-L-X-L-X-d},
    \begin{equation} \label{L-X-d-x-i}
        \sL_X dx^i = d(\sL_X x^i) = d X^i = \frac{\partial X^i}{\partial x^j}dx^j, \quad \forall i.
    \end{equation}
    Then by \eqref{def-partial-X} and Lemma \ref{lemma-nablaX-minus-nabla}, we have
    \begin{equation*}
        \begin{aligned}
            \partial^X a = &\,\,dx^i \wedge \nabla^X_{\frac{\partial}{\partial x^i}}a\\
            = &\,\,dx^i \wedge \nabla_{\frac{\partial}{\partial x^i}}a + dx^i \wedge \left[\sL_X, \nabla_{\frac{\partial}{\partial x^i}}\right]a - dx^i \wedge \nabla_{\left[X, \frac{\partial}{\partial x^i}\right]}a\\
            = &\,\,\partial a + dx^i \wedge \sL_X\left(\nabla_{\frac{\partial}{\partial x^i}}a\right) - dx^i \wedge \nabla_{\frac{\partial}{\partial x^i}}(\sL_X a) - dx^i \wedge \nabla_{\left[X, \frac{\partial}{\partial x^i}\right]}a\\
            = &\,\,\partial a - \partial(\sL_X a) + dx^i \wedge \sL_X\left(\nabla_{\frac{\partial}{\partial x^i}}a\right) - dx^i \wedge \nabla_{\left[X, \frac{\partial}{\partial x^i}\right]}a.
        \end{aligned}
    \end{equation*}
    Then by \eqref{bracket-X-partial-x-i} and \eqref{L-X-d-x-i}, we have
    \begin{equation*}
        \begin{aligned}
            \partial^X a = &\,\,\partial a - \partial(\sL_X a) + dx^i \wedge \sL_X\left(\nabla_{\frac{\partial}{\partial x^i}}a\right) - dx^i \wedge \nabla_{\left[X, \frac{\partial}{\partial x^i}\right]}a\\
            = &\,\,\partial a - \partial(\sL_X a) + dx^i \wedge \sL_X\left(\nabla_{\frac{\partial}{\partial x^i}}a\right) + \frac{\partial X^j}{\partial x^i}dx^i \wedge \nabla_{\frac{\partial}{\partial x^j}}a\\
            = &\,\,\partial a - \partial(\sL_X a) + dx^i \wedge \sL_X\left(\nabla_{\frac{\partial}{\partial x^i}}a\right) + \sL_X(dx^j) \wedge \nabla_{\frac{\partial}{\partial x^j}}a\\
            = &\,\,\partial a - \partial(\sL_X a) + \sL_X(\partial a).
        \end{aligned}
    \end{equation*}
    Therefore, we obtain
    \[[\partial, \sL_X]a = \partial(\sL_X a) - \sL_X(\partial a) = (\partial - \partial^X)a = \left[\frac{1}{\hbar}(\Gamma^X - \Gamma), a\right]. \qedhere\]
\end{proof}

\begin{remark}
    The preceding lemma is also obtained in \cite{MN} with different notations.
\end{remark}

We next record the commutator of $\sL_X$ and $D$.

\begin{proposition} \label{prop-[D-L-X]-eta-X}
    For $X \in \Gamma^{\omega}(TM)$ and $a \in \sW \otimes \wedge^{\bullet}$, we have
    \begin{equation*}
        [D, \sL_X]a = D(\sL_X a) - \sL_X(Da) = \left[\frac{1}{\hbar}\eta_X, a\right],
    \end{equation*}
    where $\eta_X \in \sW_2 \otimes \wedge^1$ is defined by \eqref{def-eta-X}.
\end{proposition}

\begin{proof}
    Since
    \[[\iota_X \Omega, a] = 0,\]
    it suffices to show
    \[[D, \sL_X ] a = \left[\frac{1}{\hbar}(\Gamma^X -\Gamma + \sL_X r), a\right].\]
    As $\sL_X$ is a derivation, we have
    \begin{equation*}
        \begin{aligned}
            \sL_X(Da) = &\,\,\sL_X\left(- \delta a + \partial a - \left[\frac{1}{\hbar}r, a\right]\right)\\
            = &\,- \sL_X(\delta a) + \sL_X(\partial a) - \left[\frac{1}{\hbar}\sL_X r, a\right] - \left[\frac{1}{\hbar}r, \sL_X a\right]\\
            = &\,-\sL_X(\delta a) + \partial(\sL_X a) - \left[\frac{1}{\hbar}r, \sL_X a\right] + [\sL_X, \partial]a - \left[\frac{1}{\hbar}\sL_X r, a\right].
        \end{aligned}
    \end{equation*}
    We claim that
    \begin{equation} \label{eq-L-X-delta-commute}
        \sL_X \circ \delta = \delta \circ \sL_X : \sW \otimes \wedge^{\bullet} \to \sW \otimes \wedge^{\bullet + 1}.
    \end{equation}
    In local Darboux coordinates $(U, x^1, \cdots, x^{2n})$, write $X = X^i \frac{\partial}{\partial x^i}$ with $X^i \in C^{\infty}(U)$. Then for $b \in \sW$, by Lemma \ref{lemma-L-iota} and \eqref{bracket-X-partial-x-i}, we have
    \[\left[\sL_X, \frac{\partial}{\partial y^i}\right]b = \left[\sL_X, \iota_{\frac{\partial}{\partial x^i}}\right]\Bigg|_{\frac{\partial}{\partial x} = \frac{\partial}{\partial y}}b = \iota_{\left[X, \frac{\partial}{\partial x^i}\right]}\Bigg|_{\frac{\partial}{\partial x} = \frac{\partial}{\partial y}}b = -\frac{\partial X^j}{\partial x^i}\frac{\partial b}{\partial y^j}\]
    for all $i$. Then we have
    \begin{equation} \label{L-X-partial-y-i}
        \begin{aligned}
            \sL_X \frac{\partial a}{\partial y^i} - \frac{\partial (\sL_X a)}{\partial y^i} = &\,\left[\sL_X, \frac{\partial}{\partial y^i}\right]a\\
            = &\,\left[{\rm Id}_{\sW} \otimes \sL_X + \sL_X \otimes {\rm Id}_{\wedge^{\bullet}}, \frac{\partial}{\partial y^i} \otimes {\rm Id}_{\wedge^{\bullet}}\right]a\\
            = &\,\left(\left[\sL_X, \frac{\partial}{\partial y^i}\right] \otimes {\rm Id}_{\wedge^{\bullet}}\right)a\\
            = &\,-\frac{\partial X^j}{\partial x^i}\frac{\partial a}{\partial y^j}.
        \end{aligned}
    \end{equation}
    Now, by \eqref{L-X-d-x-i} and \eqref{L-X-partial-y-i}, we have
    \begin{equation*}
        \begin{aligned}
            \sL_X(\delta a) = &\,\,\sL_X(d x^i) \wedge \frac{\partial a}{\partial y^i} + d x^i \wedge \sL_X\frac{\partial a}{\partial y^i}\\
            = &\,\,\frac{\partial X^i}{\partial x^j}d x^j \wedge \frac{\partial a}{\partial y^i} + d x^i \wedge \sL_X \frac{\partial a}{\partial y^i}\\
            = &\,\,dx^i \wedge \left(\frac{\partial X^j}{\partial x^i}\frac{\partial a}{\partial y^j} + \sL_X \frac{\partial a}{\partial y^i}\right)\\
            = &\,\,dx^i \wedge \frac{\partial (\sL_X a)}{\partial y^i}\\
            = &\,\,\delta(\sL_X a).
        \end{aligned}
    \end{equation*}
    Therefore, we have
    \begin{equation*}
        \begin{aligned}
            \sL_X(Da) = &\,-\sL_X(\delta a) + \partial(\sL_X a) - \left[\frac{1}{\hbar}r, \sL_X a\right] + [\sL_X, \partial]a - \left[\frac{1}{\hbar}\sL_X r, a\right]\\
            = &\,- \delta(\sL_X a) + \partial(\sL_X a) - \left[\frac{1}{\hbar}r, \sL_X a\right] + [\sL_X, \partial]a - \left[\frac{1}{\hbar}\sL_X r, a\right]\\
            = &\,\,D(\sL_X a) + [\sL_X, \partial]a - \left[\frac{1}{\hbar}\sL_X r, a\right].
        \end{aligned}
    \end{equation*}
    Then by Proposition \ref{prop-[partial-L-X]},
    \[[D, \sL_X]a = [\partial, \sL_X]a + \left[\frac{1}{\hbar}\sL_X r, a\right] = \left[\frac{1}{\hbar}(\Gamma^X - \Gamma + \sL_X r), a\right],\]
    and the proof is complete.
\end{proof}

\begin{remark}
    The identity \eqref{eq-L-X-delta-commute} is also proved in \cite{MN} as a consequence of the deformed Cartan homotopy formula \eqref{eq-deformed-Cartan-formula}.
\end{remark}

It is clear that Proposition \ref{prop-[D-L-X]-eta-X} is also valid for $\widetilde{\sW}^+$-valued differential forms. Let $X \in \Gamma^{\omega}(TM)$ be a symplectic vector field. Note that if $D u_X = -\frac{1}{\hbar}\eta_X$, then by Proposition \ref{prop-[D-L-X]-eta-X},
\begin{equation*}
    \begin{aligned}
        D(\sL_X a + [u_X, a]) = &\,\,D(\sL_X a) + [Du_X, a] + [u_X, Da]\\
        = &\,\,D(\sL_X a) - \left[\frac{1}{\hbar}\eta_X, a\right] + [u_X, Da]\\
        = &\,\,\sL_X(Da) + [u_X, Da]
    \end{aligned}
\end{equation*}
for all $a \in \sW^+ \otimes \wedge^{\bullet}$. In particular, if $a \in {\rm Ker}\, D$, then $\sL_X a + [u_X, a] \in {\rm Ker}\, D$. We next show that
\begin{equation*}
    Du_X = - \frac{1}{\hbar}\eta_X.
\end{equation*}

By Theorem \ref{thm-Fedosov-1}, we have a cochain complex
\begin{equation} \label{complex-W-D}
    \sW \otimes \wedge^{\bullet}: 0 \to \sW \xrightarrow{D} \sW \otimes \wedge^1 \xrightarrow{D} \sW \otimes \wedge^2 \xrightarrow{D} \cdots \xrightarrow{D} \sW \otimes \wedge^{2n} \to 0.
\end{equation}
Then by Theorem \ref{thm-Fedosov-2}, we have
\[H^0(\sW \otimes \wedge^{\bullet}, D) = \sW_D \cong C^{\infty}(M)[[\hbar]].\]
Fedosov proved that the cohomology vanishes in positive degrees. Moreover, given a $k$-cocycle $a$ with $k > 0$, Fedosov provided an explicit formula for some $b \in \sW \otimes \wedge^{k - 1}$ such that $D b = a$. Fedosov's original statement and proof are for the Fedosov connection $D$ with $\Omega = 0$. As we work with a general formal closed $2$-form $\Omega$ and we need this result in our main construction, we replicate the proof as follows.

\begin{theorem} \cite[Theorem 5.2.5, Corollary 5.2.6]{F2} \label{thm-cohomology-trivial}
    The cohomologies of \eqref{complex-W-D} of positive degree are all trivial. Given $a \in \sW \otimes \wedge^k$ with $k > 0$ and $Da = 0$, the following formula
    \begin{equation*}
        b = \delta^{-1}((D + \delta)b + a)
    \end{equation*}
    provides a unique solution $b \in \sW \otimes \wedge^{k - 1}$, and we have
    \begin{equation*}
        a = - D b.
    \end{equation*}
\end{theorem}

\begin{proof}
    We denote by
    \[Q : \sW \otimes \wedge^{\bullet} \to \sW \otimes \wedge^{\bullet}\]
    the exterior degree-preserving linear map defined by the recursive formula \eqref{recursive-formula-sigma-inverse-a}, i.e.,
    \begin{equation} \label{recursive-formula-Qa}
        Qa = a + \delta^{-1}(D + \delta)Qa, \quad \forall a \in \sW \otimes \wedge^{\bullet}.
    \end{equation}
    By Lemma \ref{lemma-recursive-solution-a-a-q}, \eqref{recursive-formula-Qa} admits a unique solution for any $a$. We claim that $Q$ is bijective with the inverse map
    \[Q^{-1}a = a - \delta^{-1}(D + \delta)a, \quad \forall a \in \sW \otimes \wedge^{\bullet}.\]
    Indeed, for $a \in \sW \otimes \wedge^{\bullet}$, we have
    \begin{equation*}
        \begin{aligned}
            Q^{-1}Q a = &\,\,Qa - \delta^{-1}(D + \delta)Qa\\
            = &\,\,a + \delta^{-1}(D + \delta)Qa - \delta^{-1}(D + \delta)Qa\\
            = &\,\,a.
        \end{aligned}
    \end{equation*}
    We also have
    \begin{equation*}
        \begin{aligned}
            QQ^{-1} a = &\,\,Q^{-1}a + \delta^{-1}(D + \delta)QQ^{-1}a\\
            = &\,\,a - \delta^{-1}(D + \delta)a + \delta^{-1}(D + \delta)Q^{-1}a + \delta^{-1}(D + \delta)\delta^{-1}(D + \delta)QQ^{-1}a\\
            = &\,\,a - \delta^{-1}(D + \delta)a + \delta^{-1}(D + \delta)a - \delta^{-1}(D + \delta)\delta^{-1}(D + \delta)a\\
            &\,+ \delta^{-1}(D + \delta)\delta^{-1}(D + \delta)QQ^{-1}a\\
            = &\,\,a - \delta^{-1}(D + \delta)\delta^{-1}(D + \delta)a + \delta^{-1}(D + \delta)\delta^{-1}(D + \delta)QQ^{-1}a.
        \end{aligned}
    \end{equation*}
    Hence,
    \begin{equation} \label{iteration-Q-Q-inverse-a-a}
        QQ^{-1} a - a = \delta^{-1}(D + \delta)\delta^{-1}(D + \delta)(QQ^{-1} a - a).
    \end{equation}
    Since the operator $(D + \delta)$ preserves the filtration and the operator $\delta^{-1}$ raises degrees by $1$, iterations of \eqref{iteration-Q-Q-inverse-a-a} imply that
    \[QQ^{-1} a - a \in \bigcap_{q \in \N} \sW_q \otimes \wedge^{\bullet} = \{0\}.\]
    Therefore,
    \[Q^{-1}Q a = QQ^{-1} a = a.\]
    We then prove that $Q^{-1}$ satisfies the identity
    \begin{equation} \label{eq-Q-inverse-D-delta-Q-inverse}
        Q^{-1}D + \delta Q^{-1} = 0.
    \end{equation}
    Substituting $Q^{-1}$, we see that \eqref{eq-Q-inverse-D-delta-Q-inverse} is equivalent to
    \begin{equation} \label{eq-Q-inverse-D-delta-Q-inverse-equivalent-form}
        Da - \delta^{-1}(D + \delta)Da + \delta a - \delta\delta^{-1}(D + \delta)a = 0
    \end{equation}
    for any $a \in \sW \otimes \wedge^{\bullet}$. Since $D^2 = \delta^2 = 0$, we have
    \[(D + \delta)D = \delta D = \delta(D + \delta),\]
    and \eqref{eq-Q-inverse-D-delta-Q-inverse-equivalent-form} may be written as
    \begin{equation} \label{Hodge-de Rham-decomposition-D+delta-a}
        (D + \delta)a = \delta^{-1}\delta(D + \delta)a + \delta\delta^{-1}(D + \delta)a.
    \end{equation}
    Since $((D + \delta)a)_{00} = 0$, the identity \eqref{Hodge-de Rham-decomposition-D+delta-a} is indeed the Hodge-de Rham decomposition \eqref{Hodge-de Rham-decomposition-a} of $(D + \delta)a$. Therefore, the identity \eqref{eq-Q-inverse-D-delta-Q-inverse} holds. By \eqref{eq-Q-inverse-D-delta-Q-inverse}, we have
    \begin{equation} \label{delta-D-complexes-isomorphism}
        D = - Q \delta Q^{-1},
    \end{equation}
    thus the complexes for $D$ and $- \delta$ are isomorphic. Therefore, Corollary \ref{cor-vanishing-delta-cohomology} indicates
    \[H^k(\sW \otimes \wedge^{\bullet}, D) = 0, \quad \forall k >0.\]

    Now, let $a \in \sW \otimes \wedge^k$ with $k > 0$ and $Da = 0$. By \eqref{delta-D-complexes-isomorphism}, we have
    \[\delta Q^{-1}a = 0.\]
    Then, by Lemma \ref{lemma-Hodge-de Rham-decomposition}, we have
    \[Q^{-1} a = \delta\delta^{-1}Q^{-1}a + \delta^{-1}\delta Q^{-1}a = \delta\delta^{-1}Q^{-1}a.\]
    Applying $Q$ to the preceding identity yields
    \[a = Q\delta\delta^{-1}Q^{-1}a.\]
    By \eqref{delta-D-complexes-isomorphism} again, we have
    \[a = Q\delta\delta^{-1}Q^{-1}a = - D(Q \delta^{-1}Q^{-1}a).\]
    Substituting $Q^{-1}$, the preceding identity can be written as
    \begin{equation*}
        \begin{aligned}
            a = &\,- D(Q \delta^{-1}Q^{-1}a)\\
            = &\,- D(Q\delta^{-1}(a - \delta^{-1}(D + \delta)a))\\
            = &\,- D(Q\delta^{-1}a).
        \end{aligned}
    \end{equation*}
    By the definition of $Q$, the form $Q\delta^{-1}a$ has the recursive formula
    \[Q\delta^{-1}a = \delta^{-1}a + \delta^{-1}(D + \delta)Q\delta^{-1}a,\]
    which is the same as
    \[b = \delta^{-1}((D + \delta)b + a).\]
    Therefore, the proof is complete.
\end{proof}

We next prove that $D \eta_X = 0$.

\begin{proposition} \label{prop-D-eta-X-zero}
    Let $X \in \Gamma^{\omega}(TM)$ be a symplectic vector field. Then we have
    \[D \eta_X = 0.\]
\end{proposition}

\begin{proof}
    For $a \in \sW \otimes \wedge^{\bullet}$, we have
    \[D\sL_X a - \sL_X Da = \left[\frac{1}{\hbar}\eta_X, a\right].\]
    Applying $D$ to both sides of the above equation, we have
    \begin{equation*}
        \begin{aligned}
            - D \sL_X D a = &\,\,D(D\sL_X a - \sL_X Da)\\
            = &\,\,D\left[\frac{1}{\hbar}\eta_X, a\right]\\
            = &\,\left[\frac{1}{\hbar}D\eta_X, a\right] - \left[\frac{1}{\hbar}\eta_X, Da\right].
        \end{aligned}
    \end{equation*}
    Note that
    \[D \sL_X D a = [D, \sL_X]Da = \left[\frac{1}{\hbar}\eta_X, Da\right].\]
    Therefore, we have
    \[\left[\frac{1}{\hbar}D\eta_X, a\right] = 0.\]
    Hence, $D\eta_X$ lies in the center of $\sW \otimes \wedge^{\bullet}$, so
    \[\sigma(D\eta_X) = D\eta_X.\]
    By definition, we have
    \begin{equation*}
        \begin{aligned}
            D \eta_X = &\,- \delta \eta_X + \partial \eta_X - \left[\frac{1}{\hbar}r, \eta_X\right] + D(\iota_X \Omega)\\
            = &\,- \delta \Gamma^X + \delta \Gamma - \delta \sL_X r + \partial \Gamma^X - \partial \Gamma + \partial \sL_X r\\
            &\,- \left[\frac{1}{\hbar}r, \Gamma^X\right] + \left[\frac{1}{\hbar}r, \Gamma\right] - \left[\frac{1}{\hbar}r, \sL_X r\right] + D(\iota_X \Omega).
        \end{aligned}
    \end{equation*}
    Since $\iota_X \Omega \in \hbar\wedge^1[[\hbar]]$, we have
    \[\delta(\iota_X \Omega) = \left[\frac{1}{\hbar}r, \iota_X \Omega\right] = 0.\]
    By \eqref{partial-Gamma-partial2-R} and Lemma \ref{lemma-L-iota},
    \[\partial(\iota_X \Omega) = d \iota_X \Omega = (d \iota_X + \iota_X d)\Omega = \sL_X \Omega.\]
    Therefore,
    \[D(\iota_X \Omega) = \sL_X \Omega,\]
    and
    \begin{equation*}
        \begin{aligned}
            D \eta_X = &\,- \delta \Gamma^X + \delta \Gamma - \delta \sL_X r + \partial \Gamma^X - \partial \Gamma + \partial \sL_X r\\
            &\,- \left[\frac{1}{\hbar}r, \Gamma^X\right] + \left[\frac{1}{\hbar}r, \Gamma\right] - \left[\frac{1}{\hbar}r, \sL_X r\right] + \sL_X \Omega.
        \end{aligned}
    \end{equation*}
    By Theorem \ref{thm-Fedosov-1}, we have
    \[\delta r = R + \partial r - \frac{1}{\hbar}r \circ r + \Omega.\]
    In the proof of Proposition \ref{prop-[D-L-X]-eta-X}, we have proved that
    \[\sL_X \circ \delta = \delta \circ \sL_X.\]
    Hence, we have
    \begin{equation*}
        \begin{aligned}
            D \eta_X =&\,- \delta \Gamma^X + \delta \Gamma - \delta \sL_X r + \partial \Gamma^X - \partial \Gamma + \partial \sL_X r\\
            &\,- \left[\frac{1}{\hbar}r, \Gamma^X\right] + \left[\frac{1}{\hbar}r, \Gamma\right] - \left[\frac{1}{\hbar}r, \sL_X r\right] + \sL_X \Omega\\
            = &\,- \delta \Gamma^X + \delta \Gamma + \partial \Gamma^X - \partial \Gamma - \sL_X \delta r + \sL_X \partial r + [\partial, \sL_X]r\\
            &\,- \left[\frac{1}{\hbar}r, \Gamma^X\right] + \left[\frac{1}{\hbar}r, \Gamma\right] - \left[\frac{1}{\hbar}r, \sL_X r\right] + \sL_X \Omega\\
            = &\,- \delta \Gamma^X + \delta \Gamma + \partial \Gamma^X - \partial \Gamma - \sL_X\left(R + \partial r - \frac{1}{\hbar}r \circ r + \Omega\right)\\
            &\,+ \sL_X \partial r + \left[\frac{1}{\hbar}(\Gamma^X - \Gamma), r\right] - \left[\frac{1}{\hbar}r, \Gamma^X\right] + \left[\frac{1}{\hbar}r, \Gamma\right] - \left[\frac{1}{\hbar}r, \sL_X r\right] + \sL_X \Omega\\
            = &\,- \delta \Gamma^X + \delta \Gamma + \partial \Gamma^X - \partial \Gamma - \sL_X R.
        \end{aligned}
    \end{equation*}
    Here in the third equality, we use Proposition \ref{prop-[partial-L-X]}. In the last step, we use $\sL_X(r \circ r) = [r, \sL_X r]$. By definition, we have
    \[\sigma(\delta\Gamma^X) = \sigma(\delta\Gamma) = \sigma(\partial \Gamma^X) = \sigma(\partial \Gamma) = \sigma(\sL_X R) = 0.\]
    Hence, we have
    \[\sigma(D \eta_X) = 0.\]
    Therefore,
    \[D \eta_X = \sigma(D \eta_X) = 0,\]
    and the proof is complete.
\end{proof}

Now, applying Theorem \ref{thm-cohomology-trivial} to $a = \eta_X$, we have
\[b = \delta^{-1}((D + \delta)b + \eta_X),\]
hence
\[Db = - \eta_X.\]
Note that $\frac{1}{\hbar}b$ has the same recursive formula as $u_X$. Therefore, $u_X = \frac{1}{\hbar}b$, and
\[Du_X = - \frac{1}{\hbar}\eta_X.\]
In summary, we have the following proposition.

\begin{proposition} \label{prop-D-u-X-eta-X}
    Let $X \in \Gamma^{\omega}(TM)$, let $\eta_X$ be as in \eqref{def-eta-X}, and let $u_X$ be given by \eqref{def-u-X}. Then we have
    \[D u_X = - \frac{1}{\hbar}\eta_X.\]
\end{proposition}

As a corollary, we have the following main theorem of this section.

\begin{thm} \label{thm-deformation-quantization-symplectic-vector-fields}
    Let $X \in \Gamma^{\omega}(TM)$ be a symplectic vector field. Let $u_X \in \sW^+_1$ be given by \eqref{def-u-X}. Then the derivation
    \[\sL_X + [u_X, \cdot\,] : \sW^+ \to \sW^+\]
    maps $\sW_D$ to itself, and restricts to a derivation
    \begin{equation*}
        \widehat{X} = \sL_X + [u_X, \cdot\,] : \sW_D \to \sW_D.
    \end{equation*}
    Under the algebra isomorphism $(\sW_D, \circ) \cong (C^{\infty}(M)[[\hbar]], \star)$, the derivation $\widehat{X}$ corresponds to a derivation $\widetilde{X} \in {\rm Der}(C^{\infty}(M)[[\hbar]], \star)$ such that
    \[\widetilde{X}a = Xa + O(\hbar), \quad \forall a \in C^{\infty}(M).\]
    In other words, $\widetilde{X}$ has the form
    \[\widetilde{X} = X + X_1 \hbar + X_2 \hbar^2 + \cdots\]
    with $X_q \in {\rm Hom}_{\R}(C^{\infty}(M), C^{\infty}(M))$, for every $q \in \N$.
\end{thm}

\begin{proof}
    The above discussion has shown that the derivation $\sL_X + [u_X, \cdot\,]$ maps $\sW_D$ to itself. Note that by definition, we have
    \[\widetilde{X}a = \sigma\left(\widehat{X}\left(\sigma^{-1}(a)\right)\right), \quad \forall a \in C^{\infty}(M)[[\hbar]],\]
    where
    \[\sigma : \sW_D \to C^{\infty}(M)[[\hbar]]\]
    is the bijection given by Theorem \ref{thm-Fedosov-2}. Then by the definition of $\star$, we have that $\widetilde{X}$ is a derivation of the algebra $(C^{\infty}(M)[[\hbar]], \star)$. Let $a \in C^{\infty}(M)$. By definition,
    \begin{equation*}
        \begin{aligned}
            \widetilde{X}a = &\,\,\sigma\left(\widehat{X}\left(\sigma^{-1}(a)\right)\right)\\
            = &\,\,\sigma\left(Xa + \sL_X(\sigma^{-1}(a) - a) + [u_X, \sigma^{-1}(a)]\right)\\
            = &\,\,Xa + \sigma\left(\sL_X(\sigma^{-1}(a) - a) + [u_X, \sigma^{-1}(a)]\right).
        \end{aligned}
    \end{equation*}
    By \eqref{recursive-formula-sigma-inverse-a}, we have
    \[\sL_X(\sigma^{-1}(a) - a) \in \sW_1.\]
    Since $u_X \in \sW^+_1$, we have
    \[[u_X, \sigma^{-1}(a)] \in \sW^+_1.\]
    Thus,
    \[\sL_X(\sigma^{-1}(a) - a) + [u_X, \sigma^{-1}(a)] \in \sW_1.\]
    Hence,
    \[\left(\sigma\left(\sL_X(\sigma^{-1}(a) - a) + [u_X, \sigma^{-1}(a)]\right)\right)\Big|_{\hbar = 0} = 0.\]
    Thus,
    \[\widetilde{X}a = Xa + O(\hbar), \quad \forall a \in C^{\infty}(M).\]
    Therefore, the proof is complete.
\end{proof}

\begin{remark}
    For $X, Y \in \Gamma^{\omega}(TM)$, by Proposition \ref{prop-[partial-L-X]}, we have
    \[\partial - \partial^{X + Y} = [\partial, \sL_{X + Y}] = [\partial, \sL_X + \sL_Y] = \partial - \partial^X + \partial - \partial^Y.\]
    Hence, we have
    \[\Gamma^{X + Y} - \Gamma = \Gamma^X - \Gamma + \Gamma^Y - \Gamma.\]
    Therefore, the maps
    \[\Gamma^{\omega}(TM) \to \sW \otimes \wedge^1, \quad X \mapsto \eta_X,\]
    and
    \[\Gamma^{\omega}(TM) \to \sW^+, \quad X \mapsto u_X,\]
    and
    \[\Gamma^{\omega}(TM) \to {\rm Der}(\sW_D), \quad X \mapsto \widehat{X},\]
    and
    \[\Gamma^{\omega}(TM) \to {\rm Der}(C^{\infty}(M)[[\hbar]], \star), \quad X \mapsto \widetilde{X}\]
    are all linear.
\end{remark}

\section{Deformation quantization of Lie algebra actions}

In this section, we consider the problem of quantizing a Lie algebra action on a symplectic manifold by symplectic vector fields. We show that the crossed-product algebra of such an action admits a formal deformation. This formal deformation can be realized as a cocycle crossed product.

Let $\fg$ be a Lie algebra over a field $\F$, and $\sU(\fg)$ the universal enveloping algebra of $\fg$. It is well known that $\sU(\fg)$ is a Hopf algebra with the coproduct
\[\Delta : \sU(\fg) \to \sU(\fg) \otimes \sU(\fg), \quad \Delta(\xi) = 1 \otimes \xi + \xi \otimes 1, \quad \forall \xi \in \fg,\]
the counit
\[\varepsilon : \sU(\fg) \to \F, \quad \varepsilon(\xi) = 0, \quad \forall \xi \in \fg,\]
and the antipode
\[S : \sU(\fg) \to \sU(\fg), \quad S(\xi) = - \xi, \quad \forall \xi \in \fg.\]
Let $A$ be an algebra over $\F$ with unit. An action of $\fg$ on $A$ by derivations is equivalent to a left $\sU(\fg)$-module algebra structure. More precisely, if there is a Lie algebra homomorphism
\[\Phi : \fg \to {\rm Der}(A)\]
such that $\fg$ acts on $A$ via
\[\xi \triangleright a = \Phi(\xi)(a), \quad \forall \xi \in \fg, \,\, \forall a \in A,\]
then $\Phi$ can be extended (uniquely) to a left $\sU(\fg)$ action on $A$ that makes $A$ into a left $\sU(\fg)$-module algebra. In this case, we have the crossed-product algebra $A \lcross \sU(\fg)$. As a vector space, $A \lcross \sU(\fg)$ is the tensor product $A \otimes \sU(\fg)$. The product of $A \lcross \sU(\fg)$ is defined by
\[(a \otimes \xi)(a' \otimes \xi') = \sum a(\xi_{(1)} \triangleright a') \otimes \xi_{(2)}\xi', \quad \forall a,a' \in A, \,\, \forall \xi, \xi' \in \sU(\fg).\]
Here we use the Sweedler notation
\[\Delta(\xi) = \sum \xi_{(1)} \otimes \xi_{(2)}, \quad \forall \xi \in \sU(\fg)\]
for the coproduct
\[\Delta : \sU(\fg) \to \sU(\fg) \otimes \sU(\fg)\]
of the Hopf algebra $\sU(\fg)$. The crossed-product construction is valid for general Hopf algebra actions. We also need to consider the notion of cocycle crossed product of cocycle Hopf algebra actions.

\begin{definition} \cite[Definition 6.3.1]{Ma}
    Let $H$ be a Hopf algebra over $\F$. A \textit{cocycle left $H$-module algebra} consists of an algebra $A$ over $\F$ and linear maps
    \[\triangleright : H \otimes A \to A, \quad h\otimes a \mapsto h \triangleright a, \quad \forall h \in H, \,\, \forall a \in A,\]
    \[\chi : H \otimes H \to A,\]
    such that
    \[h \triangleright (ab) = \sum (h_{(1)} \triangleright a)(h_{(2)} \triangleright b), \quad h \triangleright 1 = \varepsilon(h)1,\]
    \[1 \triangleright a = a,\]
    \[\sum(h_{(1)} \triangleright (g_{(1)} \triangleright a)) \chi(h_{(2)} \otimes g_{(2)}) = \sum \chi(h_{(1)} \otimes g_{(1)})((h_{(2)}g_{(2)}) \triangleright a),\]
    \[\sum (h_{(1)} \triangleright \chi(g_{(1)} \otimes f_{(1)}))\chi(h_{(2)} \otimes (g_{(2)}f_{(2)})) = \sum \chi(h_{(1)} \otimes g_{(1)})\chi((h_{(2)}g_{(2)}) \otimes f),\]
    \[\chi(1 \otimes h) = \chi(h \otimes 1) = \varepsilon(h)1,\]
    for all $h, g, f \in H$ and all $a, b \in A$. The map $\chi$ is called a \textit{$2$-cocycle} on $H$ with values in $A$, and $\triangleright$ is called a \textit{cocycle left action}.
\end{definition}

It is clear that if $A$ is a left $H$-module algebra via the action $\triangleright$, then $A$ is a cocycle left $H$-module algebra with the $2$-cocycle $\chi(h \otimes g) = \varepsilon(h)\varepsilon(g)1$, for all $h, g \in H$. Note that this $2$-cocycle $\chi$ is the identity of the convolution algebra ${\rm Hom}_{\F}(H \otimes H, A)$. Thus the ordinary crossed-product construction extends to \textit{cocycle crossed products}.

\begin{proposition} \cite[Proposition 6.3.2]{Ma}
    Let $H$ be a Hopf algebra, and $A$ a cocycle left module algebra over $H$ with the $2$-cocycle $\chi : H \otimes H \to A$ and the cocycle action $\triangleright : H \otimes A \to A$. Then there is a cocycle crossed product $A {_{\chi}\lcross} H$ with underlying vector space $A \otimes H$ with the associative product
    \[(a \otimes h)(b \otimes g) = \sum a(h_{(1)} \triangleright b)\chi(h_{(2)} \otimes g_{(1)}) \otimes h_{(3)}g_{(2)}\]
    and the unit element $1 \otimes 1$.
\end{proposition}

\begin{remark}
    We refer to \cite{Ma, Mo} for a discussion of cocycle left $H$-module algebras.
\end{remark}

Note that $A$ is a subalgebra of $A {_{\chi}\lcross} H$ via the inclusion
\[A \to A {_{\chi}\lcross} H, \quad a \mapsto a \otimes 1, \quad \forall a \in A.\]
An alternative description of a cocycle crossed product with a \textit{convolution-invertible $2$-cocycle} is the notion of \textit{cleft extension}. An \textit{extension of $A$ by $H$} is a right $H$-comodule algebra $E$ with the coaction
\[\beta : E \to E \otimes H\]
such that $A \subset E$ is a subalgebra and
\[A = \{e \in E \, | \, \beta(e) = e \otimes 1\}.\]
Such an extension is called \textit{cleft} if there is a convolution-invertible linear map
\[j : H \to E\]
such that $j$ is a right $H$-comodule map, and $j(1) = 1$. In \cite{Mo, Ma}, it was proved that each cocycle crossed product $A {_{\chi}\lcross} H$ with $\chi$ a convolution-invertible $2$-cocycle is a cleft extension of $A$ by $H$, and each cleft extension $E$ of $A$ by $H$ is isomorphic to some $A {_{\chi}\lcross} H$. Moreover, the equivalence classes of cocycle crossed products $A {_{\chi}\lcross} H$ with a convolution-invertible $2$-cocycle $\chi$ are in one-to-one correspondence with the equivalence classes of cleft extensions of $A$ by $H$. We refer to \cite{Mo, Ma} for a detailed discussion of Hopf algebras and the cocycle crossed-product construction.

Now we consider a Fedosov manifold $(M, \omega, \nabla)$, a formal closed $2$-form $\Omega \in \hbar \wedge^2 [[\hbar]]$, and the associated Fedosov deformation quantization. Let $X, Y \in \Gamma^{\omega}(TM)$ be two symplectic vector fields. Then we have derivations $\widehat{X}, \widehat{Y}, \widehat{[X, Y]} \in {\rm Der}(\sW_D)$ given by Theorem \ref{thm-deformation-quantization-symplectic-vector-fields}. Let $u_X, u_Y, u_{[X, Y]} \in \sW^+$ be the corresponding sections in Theorem \ref{thm-deformation-quantization-symplectic-vector-fields}, i.e.,
\[\widehat{X} = \sL_X + [u_X, \cdot\,], \quad \widehat{Y} = \sL_Y + [u_Y, \cdot\,], \quad \widehat{[X, Y]} = \sL_{[X, Y]} + [u_{[X, Y]}, \cdot\,].\]
Since $[\sL_X, \sL_Y] = \sL_{[X, Y]}$, a direct computation gives
\begin{equation*}
    \begin{aligned}
        [\widehat{X}, \widehat{Y}] = &\,\,\sL_{[X, Y]} + [[u_X, u_Y] + \sL_X u_Y - \sL_Y u_X, \cdot\,]\\
        = &\,\,\widehat{[X, Y]} + [[u_X, u_Y] - u_{[X, Y]} + \sL_X u_Y - \sL_Y u_X, \cdot\,].
    \end{aligned}
\end{equation*}
Let
\[\tau : \Gamma^{\omega}(TM) \times \Gamma^{\omega}(TM) \to \sW^+\]
be defined by
\[\tau(X, Y) = [u_X, u_Y] - u_{[X, Y]} + \sL_X u_Y - \sL_Y u_X + \frac{1}{\hbar} \iota_X \iota_Y \Omega, \quad \forall X, Y \in \Gamma^{\omega}(TM).\]
It is clear that $\tau$ is bilinear and anti-symmetric.

Now we show that the image of $\tau$ lies in $\sW_D$, and the linear maps $X \mapsto \widehat{X}$ and $\tau : \wedge^2 \Gamma^{\omega}(TM) \to \sW_D$ provide a \textit{non-abelian $2$-cocycle} on $\Gamma^{\omega}(TM)$ with values in $\sW_D$.

We first express $\eta_X$ in local Darboux coordinates.

\begin{lemma} \label{lemma-eta-X-local}
    Let $X \in \Gamma^{\omega}(TM)$. In a Darboux coordinate chart $(U, x^1, \cdots, x^{2n})$, write $X = X^i \frac{\partial}{\partial x^i}$ with $X^i \in C^{\infty}(U)$. Then we have
    \[\eta_X = \sL_X \Gamma + \sL_X r + \iota_X \Omega + \frac{1}{2}\omega_{il}\frac{\partial^2 X^l}{\partial x^j \partial x^k}y^iy^jdx^k.\]
\end{lemma}

\begin{proof}
    Let $\left(\Gamma^X\right)^k_{ij}$ be the Christoffel symbols of the connection $\nabla^X$ with respect to the coordinates $(x^1, \cdots, x^{2n})$, i.e.,
    \[\nabla^X_{\frac{\partial}{\partial x^i}}\frac{\partial}{\partial x^j} = \left(\Gamma^X\right)^k_{ij}\frac{\partial}{\partial x^k}.\]
    Then by \eqref{bracket-X-partial-x-i}, we have
    \[\left(\Gamma^X\right)^l_{jk} = \Gamma^l_{jk} + X\left(\Gamma^l_{jk}\right) - \frac{\partial X^l}{\partial x^p}\Gamma^p_{jk} + \frac{\partial X^p}{\partial x^j}\Gamma^l_{pk} + \frac{\partial X^p}{\partial x^k}\Gamma^l_{jp} + \frac{\partial^2 X^l}{\partial x^j \partial x^k}.\]
    So we have
    \begin{equation*}
        \begin{aligned}
            \Gamma^X_{ijk} = &\,\,\omega_{il}\left(\Gamma^X\right)^l_{jk}\\
            = &\,\,\Gamma_{ijk} + X(\Gamma_{ijk}) - \omega_{il}\frac{\partial X^l}{\partial x^p}\Gamma^p_{jk} + \frac{\partial X^p}{\partial x^j}\Gamma_{ipk} + \frac{\partial X^p}{\partial x^k}\Gamma_{ijp} + \omega_{il}\frac{\partial^2 X^l}{\partial x^j \partial x^k}.
        \end{aligned}
    \end{equation*}
    Since $\sL_X \omega = 0$, we have
    \[\omega_{il}\frac{\partial X^l}{\partial x^p} = \omega_{pl}\frac{\partial X^l}{\partial x^i}.\]
    Thus
    \begin{equation*}
        \begin{aligned}
            \Gamma^X_{ijk} = &\,\,\Gamma_{ijk} + X(\Gamma_{ijk}) - \omega_{il}\frac{\partial X^l}{\partial x^p}\Gamma^p_{jk} + \frac{\partial X^p}{\partial x^j}\Gamma_{ipk} + \frac{\partial X^p}{\partial x^k}\Gamma_{ijp} + \omega_{il}\frac{\partial^2 X^l}{\partial x^j \partial x^k}\\
            = &\,\,\Gamma_{ijk} + X(\Gamma_{ijk}) - \omega_{pl}\frac{\partial X^l}{\partial x^i}\Gamma^p_{jk} + \frac{\partial X^p}{\partial x^j}\Gamma_{ipk} + \frac{\partial X^p}{\partial x^k}\Gamma_{ijp} + \omega_{il}\frac{\partial^2 X^l}{\partial x^j \partial x^k}\\
            = &\,\,\Gamma_{ijk} + X(\Gamma_{ijk}) + \frac{\partial X^l}{\partial x^i}\Gamma_{ljk} + \frac{\partial X^p}{\partial x^j}\Gamma_{ipk} + \frac{\partial X^p}{\partial x^k}\Gamma_{ijp} + \omega_{il}\frac{\partial^2 X^l}{\partial x^j \partial x^k}.
        \end{aligned}
    \end{equation*}
    Hence, we have
    \begin{equation*}
        \begin{aligned}
            \Gamma^X = &\,\,\frac{1}{2}\Gamma^X_{ijk}y^iy^jdx^k\\
            = &\,\,\Gamma + \frac{1}{2}(X(\Gamma_{ijk}))y^iy^jdx^k + \frac{1}{2}\Gamma_{ljk}\left(\frac{\partial X^l}{\partial x^i}y^i\right)y^jdx^k + \frac{1}{2}\Gamma_{ipk}y^i\left(\frac{\partial X^p}{\partial x^j}y^j\right)dx^k\\
            &\,+ \frac{1}{2}\Gamma_{ijp}y^iy^j\left(\frac{\partial X^p}{\partial x^k}dx^k\right) + \frac{1}{2}\omega_{il}\frac{\partial^2 X^l}{\partial x^j \partial x^k}y^iy^jdx^k.
        \end{aligned}
    \end{equation*}
    Since
    \begin{equation} \label{L-X-d-x-i-y-i}
        \sL_X dx^i = \frac{\partial X^i}{\partial x^j}dx^j, \quad \sL_X y^i = \frac{\partial X^i}{\partial x^j}y^j, \quad \forall i,
    \end{equation}
    we have
    \begin{equation*}
        \begin{aligned}
            \Gamma^X = &\,\,\Gamma + \frac{1}{2}(X(\Gamma_{ijk}))y^iy^jdx^k + \frac{1}{2}\Gamma_{ljk}(\sL_X y^l)y^jdx^k + \frac{1}{2}\Gamma_{ipk}y^i(\sL_X y^p)dx^k\\
            &\,+ \frac{1}{2}\Gamma_{ijp}y^iy^j(\sL_X dx^p) + \frac{1}{2}\omega_{il}\frac{\partial^2 X^l}{\partial x^j \partial x^k}y^iy^jdx^k\\
            = &\,\,\Gamma + \sL_X \Gamma + \frac{1}{2}\omega_{il}\frac{\partial^2 X^l}{\partial x^j \partial x^k}y^iy^jdx^k.
        \end{aligned}
    \end{equation*}
    Therefore,
    \begin{equation*}
        \begin{aligned}
            \eta_X = &\,\,\Gamma^X - \Gamma + \sL_X r + \iota_X \Omega\\
            = &\,\,\sL_X \Gamma + \sL_X r + \iota_X \Omega + \frac{1}{2}\omega_{il}\frac{\partial^2 X^l}{\partial x^j \partial x^k}y^iy^jdx^k,
        \end{aligned}
    \end{equation*}
    and the proof is complete.
\end{proof}

\begin{lemma} \label{lemma-u-cocycle}
    The linear map
    \[\eta : \Gamma^{\omega}(TM) \to \sW \otimes \wedge^1, \quad X \mapsto \eta_X, \quad \forall X \in \Gamma^{\omega}(TM)\]
    satisfies
    \begin{equation} \label{u-cocycle}
        \sL_X \eta_Y - \sL_Y \eta_X = \eta_{[X, Y]} + d \iota_X \iota_Y \Omega, \quad \forall X, Y \in \Gamma^{\omega}(TM).
    \end{equation}
\end{lemma}
\begin{proof}
    It suffices to show that \eqref{u-cocycle} is true in a Darboux coordinate chart $(U, x^1, \cdots, x^{2n})$. Let $X, Y \in \Gamma^{\omega}(TM)$. Write $X = X^i \frac{\partial}{\partial x^i}$, $Y = Y^i \frac{\partial}{\partial x^i}$ with $X^i, Y^i \in C^{\infty}(U)$. Then by Lemma \ref{lemma-eta-X-local}, we have
    \[\eta_X = \sL_X \Gamma + \sL_X r + \iota_X \Omega + \frac{1}{2}\omega_{il}\frac{\partial^2 X^l}{\partial x^j \partial x^k}y^iy^jdx^k,\]
    and
    \[\eta_Y = \sL_Y \Gamma + \sL_Y r + \iota_Y \Omega + \frac{1}{2}\omega_{il}\frac{\partial^2 Y^l}{\partial x^j \partial x^k}y^iy^jdx^k.\]
    Since
    \[[\sL_X, \sL_Y] = \sL_X \sL_Y - \sL_Y \sL_X = \sL_{[X, Y]},\]
    we have
    \[\sL_X(\sL_Y \Gamma + \sL_Y r) - \sL_Y(\sL_X \Gamma + \sL_X r) = \sL_{[X, Y]} \Gamma + \sL_{[X, Y]} r.\]
    By Lemma \ref{lemma-L-iota}, we have
    \begin{equation*}
        \begin{aligned}
            \sL_X(\iota_Y \Omega) - \sL_Y(\iota_X \Omega) = &\,\,\sL_X \iota_Y \Omega - d\iota_Y \iota_X \Omega - \iota_Y d \iota_X \Omega\\
            = &\,\,\sL_X \iota_Y \Omega + d \iota_X \iota_Y \Omega - \iota_Y(d \iota_X + \iota_X d)\Omega\\
            = &\,\,(\sL_X \iota_Y - \iota_Y \sL_X)\Omega + d \iota_X \iota_Y \Omega\\
            = &\,\,\iota_{[X, Y]}\Omega + d \iota_X \iota_Y \Omega.
        \end{aligned}
    \end{equation*}
    Set
    \[\lambda_X = \frac{1}{2}\omega_{il}\frac{\partial^2 X^l}{\partial x^j \partial x^k}y^iy^jdx^k, \quad \lambda_Y = \frac{1}{2}\omega_{il}\frac{\partial^2 Y^l}{\partial x^j \partial x^k}y^iy^jdx^k.\]
    It remains to show that
    \[\sL_X \lambda_Y - \sL_Y \lambda_X = \lambda_{[X, Y]} = \frac{1}{2}\omega_{il}\frac{\partial^2 [X, Y]^l}{\partial x^j \partial x^k}y^iy^jdx^k,\]
    where we use the notation
    \[[X, Y] = [X, Y]^i\frac{\partial}{\partial x^i}\]
    with $[X, Y]^i \in C^{\infty}(U)$. By \eqref{L-X-d-x-i-y-i}, we have
    \begin{equation} \label{L-X-lambda-Y}
        \begin{aligned}
            \sL_X \lambda_Y = &\,\,\frac{1}{2}\omega_{il}X^p\frac{\partial^3 Y^l}{\partial x^p \partial x^j \partial x^k}y^iy^jdx^k + \frac{1}{2}\omega_{il}\frac{\partial^2 Y^l}{\partial x^j \partial x^k}(\sL_X y^i)y^jdx^k\\
            &\,+ \frac{1}{2}\omega_{il}\frac{\partial^2 Y^l}{\partial x^j \partial x^k}y^i(\sL_X y^j)dx^k + \frac{1}{2}\omega_{il}\frac{\partial^2 Y^l}{\partial x^j \partial x^k}y^iy^j(\sL_X dx^k)\\
            = &\,\,\frac{1}{2}\omega_{il}X^p\frac{\partial^3 Y^l}{\partial x^p \partial x^j \partial x^k}y^iy^jdx^k + \frac{1}{2}\omega_{il}\frac{\partial X^i}{\partial x^p}\frac{\partial^2 Y^l}{\partial x^j \partial x^k}y^py^jdx^k\\
            &\,+ \frac{1}{2}\omega_{il}\frac{\partial X^j}{\partial x^p}\frac{\partial^2 Y^l}{\partial x^j \partial x^k}y^iy^pdx^k + \frac{1}{2}\omega_{il}\frac{\partial X^k}{\partial x^p}\frac{\partial^2 Y^l}{\partial x^j \partial x^k}y^iy^jdx^p.
        \end{aligned}
    \end{equation}
    By symmetry,
    \begin{equation} \label{L-Y-lambda-X}
        \begin{aligned}
            \sL_Y \lambda_X = &\,\,\frac{1}{2}\omega_{il}Y^p\frac{\partial^3 X^l}{\partial x^p \partial x^j \partial x^k}y^iy^jdx^k + \frac{1}{2}\omega_{il}\frac{\partial Y^i}{\partial x^p}\frac{\partial^2 X^l}{\partial x^j \partial x^k}y^py^jdx^k\\
            &\,+ \frac{1}{2}\omega_{il}\frac{\partial Y^j}{\partial x^p}\frac{\partial^2 X^l}{\partial x^j \partial x^k}y^iy^pdx^k + \frac{1}{2}\omega_{il}\frac{\partial Y^k}{\partial x^p}\frac{\partial^2 X^l}{\partial x^j \partial x^k}y^iy^jdx^p.
        \end{aligned}
    \end{equation}
    From the definition,
    \[[X, Y]^i = X^j\frac{\partial Y^i}{\partial x^j} - Y^j\frac{\partial X^i}{\partial x^j}.\]
    Thus we have
    \begin{equation} \label{partial-square-[XY]}
        \begin{aligned}
            \frac{\partial^2[X, Y]^l}{\partial x^j \partial x^k} = &\,\,X^p\frac{\partial^3 Y^l}{\partial x^p \partial x^j \partial x^k} - Y^p\frac{\partial^3 X^l}{\partial x^p \partial x^j \partial x^k}\\
            &\,+ \frac{\partial X^p}{\partial x^k}\frac{\partial^2 Y^l}{\partial x^j \partial x^p} + \frac{\partial X^p}{\partial x^j}\frac{\partial^2 Y^l}{\partial x^k \partial x^p} - \frac{\partial X^l}{\partial x^p}\frac{\partial^2 Y^p}{\partial x^j \partial x^k}\\
            &\,- \frac{\partial Y^p}{\partial x^k}\frac{\partial^2 X^l}{\partial x^j \partial x^p} - \frac{\partial Y^p}{\partial x^j}\frac{\partial^2 X^l}{\partial x^k \partial x^p} + \frac{\partial Y^l}{\partial x^p}\frac{\partial^2 X^p}{\partial x^j \partial x^k}.
        \end{aligned}
    \end{equation}
    Since $\sL_X \omega = \sL_Y \omega = 0$, we have
    \begin{equation} \label{X-Y-symplectic-local}
        \omega_{ki}\frac{\partial X^k}{\partial x^j} = \omega_{kj}\frac{\partial X^k}{\partial x^i}, \quad \omega_{ki}\frac{\partial Y^k}{\partial x^j} = \omega_{kj}\frac{\partial Y^k}{\partial x^i}, \quad \forall i, j.
    \end{equation}
    Substitution of \eqref{L-X-lambda-Y}-\eqref{X-Y-symplectic-local} yields
    \[\sL_X \lambda_Y - \sL_Y \lambda_X = \lambda_{[X, Y]},\]
    and the proof is complete.
\end{proof}

\begin{remark}
    If $\Omega = 0$, then the linear map $\eta$ is a Chevalley-Eilenberg $1$-cocycle with respect to the representation
    \[\Gamma^{\omega}(TM) \to \fgl(\sW \otimes \wedge^1), \quad X \mapsto \sL_X\]
    of $\Gamma^{\omega}(TM)$ on $\sW \otimes \wedge^1$.
\end{remark}

\begin{proposition} \label{prop-D-tau-zero}
    For $X, Y \in \Gamma^{\omega}(TM)$, $\tau(X, Y) \in \sW_D$.
\end{proposition}

\begin{proof}
    By Lemma \ref{lemma-D-W-+}, it suffices to show that
    \[D(\tau(X, Y)) = 0.\]
    By Lemma \ref{lemma-u-cocycle}, Proposition \ref{prop-[D-L-X]-eta-X} and Proposition \ref{prop-D-u-X-eta-X}, we have
    \begin{equation} \label{equ-D-L-X-u-Y-L-Y-u-X}
        \begin{aligned}
            D(\sL_X u_Y - \sL_Y u_X) = &\,\,\sL_X(Du_Y) + \left[\frac{1}{\hbar}\eta_X, u_Y\right] - \sL_Y(D u_X) - \left[\frac{1}{\hbar}\eta_Y, u_X\right]\\
            = &\,-\sL_X\left(\frac{1}{\hbar}\eta_Y\right) + \sL_Y\left(\frac{1}{\hbar}\eta_X\right) + \left[\frac{1}{\hbar}\eta_X, u_Y\right] - \left[\frac{1}{\hbar}\eta_Y, u_X\right]\\
            = &\,-\frac{1}{\hbar}\eta_{[X, Y]} - \frac{1}{\hbar}d\iota_X \iota_Y \Omega + \left[\frac{1}{\hbar}\eta_X, u_Y\right] - \left[\frac{1}{\hbar}\eta_Y, u_X\right],
        \end{aligned}
    \end{equation}
    and
    \begin{equation} \label{equ-D-u-X-u-Y-D-u-XY}
        \begin{aligned}
            D[u_X, u_Y] -D u_{[X, Y]} = &\,\,[D u_X, u_Y] + [u_X, D u_Y] - D u_{[X, Y]}\\
            = &\,-\left[\frac{1}{\hbar}\eta_X, u_Y\right] + \left[\frac{1}{\hbar}\eta_Y, u_X\right] + \frac{1}{\hbar}\eta_{[X, Y]}.
        \end{aligned}
    \end{equation}
    By definition, we have
    \[\tau(X, Y) = [u_X, u_Y] - u_{[X, Y]} + \sL_X u_Y - \sL_Y u_X + \frac{1}{\hbar}\iota_X \iota_Y \Omega.\]
    Since $\iota_X \iota_Y \Omega \in C^{\infty}(M)[[\hbar]]$, we have
    \begin{equation} \label{equ-D-iota-X-iota-Y-d-iota-X-iota-Y}
        D(\iota_X \iota_Y \Omega) = d \iota_X \iota_Y \Omega
    \end{equation}
    as in the proof of Proposition \ref{prop-D-eta-X-zero}. Therefore, \eqref{equ-D-L-X-u-Y-L-Y-u-X}-\eqref{equ-D-iota-X-iota-Y-d-iota-X-iota-Y} indicate that
    \[D(\tau(X, Y)) = 0,\]
    and the proof is complete.
\end{proof}

From now on, we view $\tau$ as a linear map
\[\tau : \wedge^2 \Gamma^{\omega}(TM) \to \sW_D.\]

\begin{proposition} \label{prop-cocycle-tau}
    The linear maps
    \[\Gamma^{\omega}(TM) \to {\rm Der}(\sW_D), \quad X \mapsto \widehat{X}, \quad \forall X \in \Gamma^{\omega}(TM)\]
    and
    \[\tau :\wedge^2 \Gamma^{\omega}(TM) \to \sW_D\]
    satisfy the conditions
    \begin{equation} \label{condition-XY-YX}
        \widehat{X}\widehat{Y}a - \widehat{Y}\widehat{X}a - \widehat{[X, Y]}a = [\tau(X, Y), a]
    \end{equation}
    and
    \begin{equation} \label{condition-tau-XY-Z}
        \tau([X, Y], Z) + \tau([Y, Z], X) + \tau([Z, X], Y) = \widehat{X}(\tau(Y, Z)) + \widehat{Y}(\tau(Z, X)) + \widehat{Z}(\tau(X, Y))
    \end{equation}
    for all $X, Y, Z \in \Gamma^{\omega}(TM)$ and all $a \in \sW_D$.
\end{proposition}
\begin{proof}
    The condition \eqref{condition-XY-YX} is clear from the definition of $\tau$.
    
    Now we show \eqref{condition-tau-XY-Z}. For simplicity, we use the notation $+ \circlearrowleft (X, Y, Z)$ to stand for the summation over cyclic permutations of $(X, Y, Z)$. We first compute
    \begin{equation*}
        \begin{aligned}
            \tau([X, Y], Z) = &\,\,[u_{[X, Y]}, u_Z] - u_{[[X, Y], Z]} + \sL_{[X, Y]}u_Z - \sL_Z u_{[X, Y]} + \frac{1}{\hbar}\iota_{[X, Y]}\iota_Z \Omega\\
            = &\,\,[u_{[X, Y]}, u_Z] - u_{[[X, Y], Z]} + \sL_{[X, Y]}u_Z - \sL_Z u_{[X, Y]} - \frac{1}{\hbar}\Omega([X, Y], Z).
        \end{aligned}
    \end{equation*}
    Since the map $X \mapsto u_X$ is linear,
    \[u_{[[X, Y], Z]} + \circlearrowleft (X, Y, Z) = 0.\]
    Therefore,
    \begin{equation} \label{tau-XY-Z-cyclic-sum}
        \begin{aligned}
            &\,\,\tau([X, Y], Z) + \circlearrowleft (X, Y, Z)\\
            = &\,\left( [u_{[X, Y]}, u_Z] + \sL_{[X, Y]}u_Z - \sL_Z u_{[X, Y]} - \frac{1}{\hbar}\Omega([X, Y], Z) \right) + \circlearrowleft (X, Y, Z).
        \end{aligned}
    \end{equation}
    We then compute
    \begin{equation*}
        \begin{aligned}
            &\,\,\widehat{X}(\tau(Y, Z))\\
            = &\,\,\sL_X [u_Y, u_Z] - \sL_X u_{[Y, Z]} + \sL_X \sL_Y u_Z - \sL_X \sL_Z u_Y - \frac{1}{\hbar}X\Omega(Y, Z)\\
            &\,+ [u_X, [u_Y, u_Z]] - [u_X, u_{[Y, Z]}] + [u_X, \sL_Y u_Z] - [u_X, \sL_Z u_Y]\\
            = &\,\,[\sL_X u_Y, u_Z] + [u_Y, \sL_X u_Z] - \sL_X u_{[Y, Z]} + \sL_X \sL_Y u_Z - \sL_X \sL_Z u_Y \\
            &\,+ [u_X, [u_Y, u_Z]] - [u_X, u_{[Y, Z]}] + [u_X, \sL_Y u_Z] - [u_X, \sL_Z u_Y] - \frac{1}{\hbar}X\Omega(Y, Z).
        \end{aligned}
    \end{equation*}
    The Jacobi identity of the commutator bracket indicates
    \[[u_X, [u_Y, u_Z]] + \circlearrowleft (X, Y, Z) = 0.\]
    A direct computation gives
    \[([\sL_X u_Y, u_Z] + [u_Y, \sL_X u_Z] + [u_X, \sL_Y u_Z] - [u_X, \sL_Z u_Y]) + \circlearrowleft (X, Y, Z) = 0,\]
    and
    \begin{equation*}
        \begin{aligned}
            (\sL_X \sL_Y u_Z - \sL_X \sL_Z u_Y) + \circlearrowleft (X, Y, Z) = &\,\,[\sL_X, \sL_Y]u_Z + \circlearrowleft(X, Y, Z)\\
            = &\,\,\sL_{[X, Y]}u_Z + \circlearrowleft (X, Y, Z).
        \end{aligned}
    \end{equation*}
    Therefore,
    \begin{equation} \label{widehat-X-tau-XY-cyclic-sum}
        \begin{aligned}
            &\,\,\widehat{X}(\tau(Y, Z)) + \circlearrowleft (X, Y, Z)\\
            = &\,\left(-\sL_X u_{[Y, Z]} + \sL_{[X, Y]} u_Z - [u_X, u_{[Y, Z]}] - \frac{1}{\hbar}X\Omega(Y, Z)\right) + \circlearrowleft (X, Y, Z)\\
            = &\,\left([u_{[X, Y]}, u_Z] + \sL_{[X, Y]} u_Z - \sL_Z u_{[X, Y]} - \frac{1}{\hbar}X\Omega(Y, Z)\right) + \circlearrowleft (X, Y, Z).
        \end{aligned}
    \end{equation}
    Since $\Omega$ is closed, we have
    \begin{equation*}
        \begin{aligned}
            0 = &\,\,(d\Omega)(X, Y, Z)\\
            = &\,\,X\Omega(Y, Z) - Y\Omega(X, Z) + Z\Omega(X, Y)\\
            &\,- \Omega([X, Y], Z) + \Omega([X, Z], Y) - \Omega([Y, Z], X).
        \end{aligned}
    \end{equation*}
    Hence,
    \begin{equation} \label{Omega-XY-Z-X-Omega-YZ-cyclic-sum}
        \Omega([X, Y], Z) + \circlearrowleft (X, Y, Z) = X\Omega(Y, Z) + \circlearrowleft (X, Y, Z).
    \end{equation}
    Substitution of \eqref{tau-XY-Z-cyclic-sum}-\eqref{Omega-XY-Z-X-Omega-YZ-cyclic-sum} yields the condition \eqref{condition-tau-XY-Z}.
\end{proof}

For later use, we give the following simple lemma.

\begin{lemma} \label{lemma-tau-hbar-zero}
    Let
    \[\tau :\wedge^2 \Gamma^{\omega}(TM) \to \sW_D\]
    be given as above. Then
    \[(\sigma(\tau(X, Y)))\Big|_{\hbar = 0} = \left(\frac{1}{\hbar}\iota_X \iota_Y \Omega\right)\Big|_{\hbar = 0}, \quad \forall X, Y \in \Gamma^{\omega}(TM).\]
\end{lemma}
\begin{proof}
    Let $X, Y \in \Gamma^{\omega}(TM)$. Since $u_X, u_Y, u_{[X, Y]} \in \sW^+_1$, we have
    \[[u_X, u_Y] - u_{[X, Y]} + \sL_X u_Y - \sL_Y u_X \in \sW^+_1.\]
    Then, since $\frac{1}{\hbar}\iota_X \iota_Y \Omega \in C^{\infty}(M)[[\hbar]] \subset \sW$ and $\tau(X, Y) \in \sW_D \subset \sW$, we have
    \[[u_X, u_Y] - u_{[X, Y]} + \sL_X u_Y - \sL_Y u_X \in \sW_1.\]
    Hence,
    \[(\sigma([u_X, u_Y] - u_{[X, Y]} + \sL_X u_Y - \sL_Y u_X))\Big|_{\hbar = 0} = 0.\]
    Therefore,
    \[(\sigma(\tau(X, Y)))\Big|_{\hbar = 0} = \left(\sigma\left(\frac{1}{\hbar}\iota_X \iota_Y \Omega\right)\right)\Big|_{\hbar = 0} = \left(\frac{1}{\hbar}\iota_X \iota_Y \Omega\right)\Big|_{\hbar = 0},\]
    and the proof is complete.
\end{proof}

The conditions \eqref{condition-XY-YX} and \eqref{condition-tau-XY-Z} in Proposition \ref{prop-cocycle-tau} motivate the following abstract definition.

\begin{definition} \cite[Example 7.1.7]{Mo} \label{def-non-abelian-2-cocycle-Lie}
    Let $\fg$ be a Lie algebra over a field $\F$. Let $A$ be an algebra over $\F$. A \textit{cocycle action} of $\fg$ on $A$ is a linear map
    \[\rho : \fg \to {\rm Der}(A), \quad x \mapsto \rho_x, \quad \forall x \in \fg,\]
    together with a linear map
    \[\tau: \wedge^2 \fg \to A\]
    such that
    \begin{equation} \label{condition-delta-cocycle}
        \rho_x\rho_y a - \rho_y\rho_x a - \rho_{[x, y]}a = [\tau(x, y), a],
    \end{equation}
    and
    \begin{equation} \label{condition-tau-cocycle}
        \tau([x, y], z) + \tau([y, z], x) + \tau([z, x], y) = \rho_x(\tau(y, z)) + \rho_y(\tau(z, x)) + \rho_z(\tau(x, y))
    \end{equation}
    for all $x, y, z \in \fg$ and all $a \in A$. The map $\tau$ is called a \textit{non-abelian $2$-cocycle} on $\fg$ with values in $A$, and $\rho$ is called a \textit{cocycle action}.
\end{definition}

\begin{remark}
    If $A$ is commutative, then the condition \eqref{condition-delta-cocycle} is equivalent to the statement that
    \[\rho : \fg \to {\rm Der}(A)\]
    is a Lie algebra homomorphism. Then the condition \eqref{condition-tau-cocycle} is equivalent to saying that $\tau$ is a Chevalley-Eilenberg $2$-cocycle with respect to the representation
    \[\rho : \fg \to {\rm Der}(A).\]
\end{remark}

We now assume that $\F$ has characteristic zero. It was proved in \cite{Mo} that a cocycle action $(\rho, \tau)$ of $\fg$ on $A$ provides a cleft extension of $A$ by $\sU(\fg)$. We now briefly recall the construction. First, consider
\begin{equation} \label{Lie-bracket-cocycle}
    [(a, x), (b, y)] \coloneqq ([a, b] + \rho_x(b) - \rho_y(a) + \tau(x, y), [x, y]), \quad \forall a, b \in A, \,\, \forall x, y \in \fg.
\end{equation}
A direct verification shows that the displayed bracket is a Lie bracket. Hence we obtain a Lie algebra structure on $A \oplus \fg$ with the above Lie bracket. We denote this Lie algebra by $A \times_{\tau} \fg$. Then we have an exact sequence of Lie algebras
\[0 \to A \to A \times_{\tau} \fg \to \fg \to 0.\]
Here $A$ is regarded as a Lie algebra under the commutator bracket. The map $A \to A \times_{\tau} \fg$ is the natural inclusion, and the map $A \times_{\tau} \fg \to \fg$ is the natural projection. Let $I \subset \sU(A \times_{\tau} \fg)$ be the ideal of $\sU(A \times_{\tau} \fg)$ generated by the subset $S = \{1_{\sU(A)} - 1_A, a \bullet b - ab \, | \, a, b \in A\}$. It is clear that $1_{\sU(A)} = 1_{\sU(A \times_{\tau} \fg)}$ via the natural inclusion $\sU(A) \subset \sU(A \times_{\tau} \fg)$ of algebras. The notation $1_A$ denotes the unit of $A$ as an algebra. The notation $a \cdot b$ means the product of $a$ and $b$ in $\sU(A) \subset \sU(A \times_{\tau} \fg)$, and the notation $ab$ denotes the algebra product of $a$ and $b$ in the algebra $A$. Then let
\[A \lcross_{\tau} \sU(\fg) \coloneqq \sU(A \times_{\tau} \fg) / I.\]
Set $I_A = I \cap \sU(A)$. It is clear that $I_A$ is an ideal of $\sU(A)$. Moreover, for $a \in A$ and $x \in \fg$, \eqref{Lie-bracket-cocycle} indicates
\begin{equation} \label{eq-[x,a]}
    x \bullet a = a \bullet x + \rho_x(a)
\end{equation}
in $\sU(A \times_{\tau} \fg)$. Let $B = \{x_{\alpha}\}$ be an ordered basis of $\fg$. Set
\[\sB = \{x_{\alpha_1} \bullet \cdots \bullet x_{\alpha_k} \, | \, x_{\alpha_1} \leq \cdots \leq x_{\alpha_k} \in B, \,\, k \in \N\}.\]
By \eqref{eq-[x,a]}, we have
\[\sB \bullet S \subset S \bullet \sB + {\rm Span}_{\F}S.\]
Note that under the natural inclusion $\sU(A) \subset \sU(A \times_{\tau} \fg)$, $S$ is a subset of $\sU(A)$. Then, according to the PBW theorem, $I_A$ is generated by $S \subset \sU(A)$. Therefore, there is an algebra isomorphism
\[\sU(A)/ I_A \cong A.\]

It was proved in \cite{Mo} that given a short exact sequence
\[0 \to \fh \to \fg \to \fg / \fh \to 0\]
of Lie algebras, the universal enveloping algebra $\sU(\fg)$ is a cleft extension of $\sU(\fh)$ by $\sU(\fg / \fh)$. Hence we have
\[\sU(\fg) \cong \sU(\fh) {_{\chi}\lcross} \sU(\fg / \fh),\]
where $\chi : \sU(\fg / \fh) \otimes \sU(\fg / \fh) \to \sU(\fh)$ is a convolution-invertible $2$-cocycle for the associated cocycle left action. Applying the above discussion to the exact sequence
\[0 \to A \to A \times_{\tau} \fg \to \fg \to 0,\]
we have that
\[\sU(A \times_{\tau} \fg) \cong \sU(A) {_{\chi'}\lcross} \sU(\fg)\]
for some convolution-invertible $2$-cocycle $\chi' : \sU(\fg) \otimes \sU(\fg) \to \sU(A)$. Moreover, the explicit proof in \cite{Mo} indicates that the above isomorphism is compatible with the algebra isomorphism $\sU(A) / I_A \cong A$. Therefore, we have
\[A \lcross_{\tau} \sU(\fg) \cong A {_{\chi}\lcross} \sU(\fg)\]
for some convolution-invertible $2$-cocycle $\chi : \sU(\fg) \otimes \sU(\fg) \to A$.

Applying the preceding discussion of abstract cocycle actions to the quantization of symplectic vector fields in Theorem \ref{thm-deformation-quantization-symplectic-vector-fields}, we obtain the following main result of this paper.

\begin{theorem} \label{thm-main-theorem-deformation-cross-product}
    Let $(M, \omega, \nabla)$ be a Fedosov manifold, and let $\Omega \in \hbar \wedge^2[[\hbar]]$ be a series of closed $2$-forms. Let $\fg$ be a Lie algebra acting on $C^{\infty}(M)$ via a Lie algebra homomorphism
    \[\Phi : \fg \to \Gamma^{\omega}(TM),\]
    i.e.,
    \[x \triangleright f = \Phi(x)f, \quad \forall x \in \fg, \,\, \forall f \in C^{\infty}(M).\]
    Let $(C^{\infty}(M)[[\hbar]], \star)$ be the Fedosov deformation quantization given by \eqref{star-product-Fedosov}. There exists a cocycle crossed product $(C^{\infty}(M)[[\hbar]] {_{\chi} \lcross} \sU(\fg), \tilde{\star})$ of the algebra $(C^{\infty}(M)[[\hbar]], \star)$. If
    \begin{equation} \label{Omega-x-y-vanishes-first}
        \left(\frac{1}{\hbar}\Omega(\Phi(x), \Phi(y))\right)\Big|_{\hbar = 0} = 0, \quad \forall x, y \in \fg,
    \end{equation}
    the product $\tilde{\star}$ is naturally extended to a formal deformation $((C^{\infty}(M) \lcross \sU(\fg)) [[\hbar]], \tilde{\star})$ of the crossed-product algebra $C^{\infty}(M) \lcross \sU(\fg)$.
\end{theorem}

\begin{proof}
    Define a linear map
    \[\tilde{\rho} : \fg \to {\rm Der}(C^{\infty}(M)[[\hbar]], \star), \quad x \mapsto \tilde{\rho}_x, \quad \forall x \in \fg\]
    by
    \[\tilde{\rho}_x = \widetilde{\Phi(x)}, \quad \forall x \in \fg,\]
    where $\widetilde{\Phi(x)}$ is the derivation constructed in Theorem \ref{thm-deformation-quantization-symplectic-vector-fields}. Let
    \[\tau : \wedge^2 \Gamma^{\omega}(TM) \to \sW_D\]
    be the linear map in Proposition \ref{prop-cocycle-tau}. Define a linear map
    \[\tilde{\tau} : \wedge^2 \fg \to C^{\infty}(M)[[\hbar]]\]
    by
    \[\tilde{\tau}(x, y) = \sigma(\tau(\Phi(x), \Phi(y))), \quad \forall x, y \in \fg,\]
    where
    \[\sigma : \sW_D \to C^{\infty}(M)[[\hbar]]\]
    is the isomorphism in the definition of $\star$. Then by Proposition \ref{prop-cocycle-tau}, we have that $\tilde{\rho}$ and $\tilde{\tau}$ provide a non-abelian $2$-cocycle on $\fg$ with values in $(C^{\infty}(M)[[\hbar]], \star)$. Then the above discussion on cocycle crossed product indicates that we have a cocycle crossed product
    \[C^{\infty}(M)[[\hbar]] {_{\chi} \lcross} \sU(\fg)\]
    of the algebra $(C^{\infty}(M)[[\hbar]], \star)$. Denote this product by $\tilde{\star}$. By Theorem \ref{thm-deformation-quantization-symplectic-vector-fields},
    \[\widetilde{\Phi(x)}|_{\hbar = 0} = \Phi(x), \quad \forall x \in \fg.\]
    By Lemma \ref{lemma-tau-hbar-zero},
    \[\tilde{\tau}(x, y)|_{\hbar = 0} = - \left(\frac{1}{\hbar}\Omega(\Phi(x), \Phi(y))\right)\Big|_{\hbar = 0}, \quad \forall x, y \in \fg.\]
    Now assume that \eqref{Omega-x-y-vanishes-first} holds, we then have
    \[\tilde{\tau}(x, y)|_{\hbar = 0} = 0, \quad \forall x, y \in \fg.\]
    Thus, at $\hbar = 0$, $\tilde{\star}$ reduces to the multiplication in $C^{\infty}(M) \lcross \sU(\fg)$. Moreover, there exist natural inclusions
    \[C^{\infty}(M) \otimes \sU(\fg) \subset C^{\infty}(M)[[\hbar]] \otimes \sU(\fg) \subset (C^{\infty}(M) \otimes \sU(\fg))[[\hbar]]\]
    of vector spaces. As an $\R[[\hbar]]$-bilinear product of $(C^{\infty}(M) \lcross \sU(\fg))[[\hbar]]$ is completely determined by its values taken on elements of $C^{\infty}(M) \lcross \sU(\fg)$, the star product $\tilde{\star}$ is naturally extended to a formal deformation $((C^{\infty}(M) \lcross \sU(\fg)) [[\hbar]], \tilde{\star})$ of $C^{\infty}(M) \lcross \sU(\fg)$.
\end{proof}

The following corollary follows immediately from the classification of deformation quantization.

\begin{corollary}
    Let $(M, \omega, \nabla)$ be a Fedosov manifold, and let $\Omega \in \hbar \wedge^2[[\hbar]]$ be a series of closed $2$-forms. Let $\fg$ be a Lie algebra acting on $C^{\infty}(M)$ by symplectic vector fields. Let $(C^{\infty}(M)[[\hbar]], \star)$ be the Fedosov deformation quantization given by \eqref{star-product-Fedosov}. If the closed $2$-form $\left(\frac{1}{\hbar}\Omega\right)\Big|_{\hbar = 0}$ is exact, there is a formal deformation $((C^{\infty}(M) \lcross \sU(\fg)) [[\hbar]], \tilde{\star})$ of the crossed-product algebra $C^{\infty}(M) \lcross \sU(\fg)$ which coincides with $\star$ on $C^{\infty}(M)[[\hbar]]$ via the inclusion
    \[C^{\infty}(M)[[\hbar]] \to (C^{\infty}(M) \lcross \sU(\fg)) [[\hbar]], \quad f \mapsto f \otimes 1, \quad \forall f \in C^{\infty}(M)[[\hbar]].\]
\end{corollary}

\begin{proof}
    Write $\Omega = \sum\limits_{q = 1}^{\infty}\Omega_q \hbar^q$ with each $\Omega_q \in \wedge^2$ a closed $2$-form. By definition, $\left(\frac{1}{\hbar}\Omega\right)\Big|_{\hbar = 0} = \Omega_1$. Now assume that $\Omega_1$ is exact. Set
    \[\Omega' \coloneqq \Omega - \hbar \Omega_1 = \sum_{q = 2}^{\infty} \Omega_q \hbar^q \in \wedge^2 [[\hbar]].\]
    Let $(C^{\infty}(M)[[\hbar]], \star')$ be the Fedosov deformation quantization of $(M, \omega, \nabla)$ and $\Omega'$. By Remark \ref{rmk-classification}, $\star$ and $\star'$ are equivalent. Therefore, applying Theorem \ref{thm-main-theorem-deformation-cross-product} to $\star'$ and the action of $\fg$ yields the corollary.
\end{proof}

Since each Fedosov manifold $(M, \omega, \nabla)$ admits a canonical Fedosov deformation quantization, i.e., the deformation quantization \eqref{star-product-Fedosov} with $\Omega = 0$, we have the following corollary.

\begin{corollary}
    Let $(M, \omega)$ be a symplectic manifold with the canonical Poisson bracket $\{\cdot \,, \cdot\}$, and let $\fg$ be a Lie algebra acting on $C^{\infty}(M)$ by symplectic vector fields. There exists a formal deformation $((C^{\infty}(M) \lcross \sU(\fg)) [[\hbar]], \tilde{\star})$ of the crossed-product algebra $C^{\infty}(M) \lcross \sU(\fg)$, such that
    \[\frac{(f \otimes 1) \tilde{\star} (g \otimes 1) - (g \otimes 1) \tilde{\star} (f \otimes 1)}{\hbar}\Big|_{\hbar = 0} = \{f, g\} \otimes 1, \quad \forall f, g \in C^{\infty}(M).\]
\end{corollary}

\begin{proof}
    Choose a symplectic connection $\nabla$ on $(M, \omega)$. Applying Theorem \ref{thm-main-theorem-deformation-cross-product} to $(M, \omega, \nabla)$ and the $\fg$-action with $\Omega = 0$ yields the corollary.
\end{proof}

We end this section with some remarks.

\begin{remark}
    Given a cocycle Lie algebra action
    \[\rho : \fg \to {\rm Der}(A)\]
    with a non-abelian $2$-cocycle
    \[\tau : \wedge^2 \fg \to A,\]
    \cite[Proposition 7.2.3]{Mo} and the proof of \cite[Corollary 7.2.8]{Mo} provide an explicit construction of a cocycle left $\sU(\fg)$-action on $A$ with a $2$-cocycle
    \[\chi : \sU(\fg) \otimes \sU(\fg) \to A,\]
    such that
    \[A \lcross_{\tau} \sU(\fg) \cong A {_{\chi}\lcross} \sU(\fg).\]
    This construction depends on the choice of an ordered basis of $\fg$. Moreover, one can show that
    \[\chi(x \otimes y - y \otimes x) = \tau(x, y), \quad \forall x, y \in \fg.\]
    Therefore, the cocycle crossed product $(C^{\infty}(M)[[\hbar]] {_{\chi}\lcross} \sU(\fg), \tilde{\star})$ in Theorem \ref{thm-main-theorem-deformation-cross-product} does not reduce to the crossed product $C^{\infty}(M) \lcross \sU(\fg)$ at $\hbar = 0$ in general.
\end{remark}

\begin{remark}
    In the proof of Theorem \ref{thm-main-theorem-deformation-cross-product}, we use inclusions
    \[C^{\infty}(M) \otimes \sU(\fg) \subset C^{\infty}(M)[[\hbar]] \otimes \sU(\fg) \subset (C^{\infty}(M) \otimes \sU(\fg))[[\hbar]]\]
    to extend the cocycle crossed product $(C^{\infty}(M)[[\hbar]] {_{\chi}\lcross} \sU(\fg), \tilde{\star})$ to a formal deformation  $((C^{\infty}(M) \lcross \sU(\fg)) [[\hbar]], \tilde{\star})$ under the assumption \eqref{Omega-x-y-vanishes-first}. The choice of a cocycle left $\sU(\fg)$-action with a $2$-cocycle $\chi$ in this construction is not unique. However, for different choices with $2$-cocycles $\chi$ and $\chi'$, we always have
    \begin{equation} \label{isomorphism-cocycle-cross-product-chi-chi'}
        C^{\infty}(M)[[\hbar]] {_{\chi}\lcross} \sU(\fg) \cong C^{\infty}[[\hbar]] \lcross_{\tilde{\tau}} \sU(\fg) \cong C^{\infty}(M)[[\hbar]] {_{\chi'}\lcross} \sU(\fg)
    \end{equation}
    as the construction of the cleft extension $C^{\infty}[[\hbar]] \lcross_{\tilde{\tau}} \sU(\fg)$ is canonical. Moreover, the isomorphism
    \[C^{\infty}(M)[[\hbar]] {_{\chi}\lcross} \sU(\fg) \cong C^{\infty}(M)[[\hbar]] {_{\chi'}\lcross} \sU(\fg)\]
    in \eqref{isomorphism-cocycle-cross-product-chi-chi'} extends naturally to an isomorphism of formal deformations. Another way to obtain the formal deformation in Theorem \ref{thm-main-theorem-deformation-cross-product} is to use completed cleft extensions and cocycle crossed products in the preceding construction. Indeed, one can use the completed tensor product $\widehat{\otimes} = \widehat{\otimes}_{\R[[\hbar]]}$ with respect to the $\hbar$-adic topology to define completed cocycle crossed products and cleft extensions over the topological algebra $\R[[\hbar]]$. Then, all constructions in this section are also valid for the Lie algebra $\Gamma^{\omega}(TM)[[\hbar]]$ of formal symplectic vector fields. In particular, the cocycle crossed product $(C^{\infty}(M)[[\hbar]] {_{\chi}\lcross} \sU(\fg), \tilde{\star})$ in Theorem \ref{thm-main-theorem-deformation-cross-product} can be improved to a completed cocycle crossed product
    \begin{equation} \label{completed-cocycle-cross-product}
        (C^{\infty}(M)[[\hbar]] {_{\chi}\widehat{\lcross}} \sU(\fg)[[\hbar]], \tilde{\star}).
    \end{equation}
    As a vector space, we have
    \[C^{\infty}(M)[[\hbar]] {_{\chi}\widehat{\lcross}} \sU(\fg)[[\hbar]] = (C^{\infty}(M) \otimes \sU(\fg))[[\hbar]].\]
    Thus, if the condition \eqref{Omega-x-y-vanishes-first} holds, the completed cocycle crossed product \eqref{completed-cocycle-cross-product} is in fact the formal deformation in Theorem \ref{thm-main-theorem-deformation-cross-product}.
\end{remark}

\appendix
\section{M\"uller-Bahns and Neumaier's construction of derivations}
In this appendix, we briefly recall M\"uller-Bahns and Neumaier's construction of derivations of Fedosov deformation quantizations. In particular, given a Fedosov manifold $(M, \omega, \nabla)$ and a series of closed $2$-forms $\Omega \in \hbar \wedge^2 [[\hbar]]$, we show that the derivation $\widetilde{X}$ in Theorem \ref{thm-deformation-quantization-symplectic-vector-fields} is the same as the derivation in M\"uller-Bahns and Neumaier's construction with respect to the closed $1$-form $- \iota_X \omega$.

Let $(M, \omega, \nabla)$ be a Fedosov manifold, and let $\Omega \in \hbar \wedge^2 [[\hbar]]$ be a series of closed $2$-forms. Let $A \in \wedge^1[[\hbar]]$ be a series of $1$-forms. As in the proof of Proposition \ref{prop-D-eta-X-zero}, we have $DA = dA$. Thus, the formal $1$-form $A$ is $D$-closed if and only if $A$ is closed. Now assume $A$ is closed. By Theorem \ref{thm-cohomology-trivial}, we have $h_A \in \sW$ given by
\begin{equation} \label{eq-formula-h-A}
    h_A = \delta^{-1}((D + \delta)h_A + A),
\end{equation}
such that
\[A = - D h_A.\]
Set
\[D_{h_A} \coloneqq \left[\frac{1}{\hbar}h_A, \cdot\,\right]\]
We then have
\[[D, D_{h_A}] = \left[\frac{1}{\hbar}D h_A, \cdot \,\right] = \left[- \frac{1}{\hbar}A, \cdot \,\right] = 0.\]
Therefore, $D_{h_A}$ restricts to a derivation of $\sW_D$. Moreover, M\"uller-Bahns and Neumaier proved the following results.

\begin{proposition} \cite[Lemma 2.1, Proposition 2.2]{MN} \label{prop-classification-derivations-star}
    Let $Z^1_{\rm dR}(M)$ be the space of closed $1$-forms on $M$. Let $\star$ be the Fedosov deformation quantization of $(M, \omega, \nabla)$ and $\Omega$.
    \begin{itemize}
        \item [(i)] For any series of closed $1$-forms $A \in Z^1_{\rm dR}(M)[[\hbar]]$, there exists a uniquely determined element $h_A \in \sW$ such that $- D h_A = A$ and $\sigma(h_A) = 0$. Moreover, $h_A$ is explicitly given by \eqref{eq-formula-h-A}.
        \item [(ii)] For any $A \in Z^1_{\rm dR}(M)[[\hbar]]$, the map
        \[D_A : C^{\infty}(M)[[\hbar]] \to C^{\infty}(M)[[\hbar]]\]
        defined by
        \[D_A f \coloneqq \sigma(D_{h_A}\sigma^{-1}(f)) = \sigma\left(\left[\frac{1}{\hbar}h_A, \sigma^{-1}(f)\right]\right), \quad \forall f \in C^{\infty}(M)[[\hbar]]\]
        is an $\R[[\hbar]]$-linear derivation of $(C^{\infty}(M)[[\hbar]], \star)$. Hence, this construction yields a linear map
        \[Z^1_{\rm dR}(M)[[\hbar]] \to {\rm Der}(C^{\infty}(M)[[\hbar]], \star), \quad A \mapsto D_A.\]
        \item [(iii)] The above map
        \[Z^1_{\rm dR}(M)[[\hbar]] \to {\rm Der}(C^{\infty}(M)[[\hbar]], \star), \quad A \mapsto D_A\]
        is bijective. Moreover, for any $f \in C^{\infty}(M)[[\hbar]]$, we have that $D_{df}$ is a quasi-inner derivation, i.e.,
        \[D_{df}(g) = \left[\frac{1}{\hbar}f, g\right]_{\star} = \frac{1}{\hbar}(f \star g - g \star f), \quad \forall g \in C^{\infty}(M)[[\hbar]].\]
        The induced map
        \[[A] \mapsto [D_A]\]
        from $H^1_{\rm dR}(M)[[\hbar]]$ to ${\rm Der}(C^{\infty}(M)[[\hbar]], \star) / {\rm Der}^{\rm qi}(C^{\infty}(M)[[\hbar]], \star)$, the space of derivations of $\star$ modulo the quasi-inner derivations, is also bijective.
    \end{itemize}
\end{proposition}

\begin{remark}
    The recursive formula \eqref{eq-formula-h-A} of $h_A$ only implies that $h_A \in \sW_1$. Hence $\frac{1}{\hbar}h_A \notin \sW^+$ in general.
\end{remark}

Let $X \in \Gamma^{\omega}(TM)$. By Lemma \ref{lemma-L-iota}, we have
\[d(\iota_X \omega) = (d\iota_X + \iota_X d)\omega = \sL_X \omega = 0.\]
Thus, we have a derivation $D_{- \iota_X \omega}$ of $(C^{\infty}(M)[[\hbar]], \star)$ by Proposition \ref{prop-classification-derivations-star}. Using the recursive formula \eqref{recursive-formula-sigma-inverse-a}, one can see that
\[D_{- \iota_X \omega}|_{\hbar = 0} = X.\]
We now show that
\[D_{- \iota_X \omega} = \widetilde{X},\]
where $\widetilde{X}$ is the derivation in Theorem \ref{thm-deformation-quantization-symplectic-vector-fields}.

We recall the deformed Cartan homotopy formula, which is frequently used in \cite{MN}.
\begin{lemma} \cite[Lemma 3.1]{MN} \label{lemma-deformed-Cartan-formula}
    Let $X \in \Gamma^{\omega}(TM)$. Set
    \[\theta_X = \delta^{-1}(\iota_X \omega).\]
    The derivation
    \[\sL_X : \sW \otimes \wedge^{\bullet} \to \sW \otimes \wedge^{\bullet}\]
    can be expressed as
    \begin{equation} \label{eq-deformed-Cartan-formula}
        \begin{aligned}
            \sL_X = &\,\,[D, \iota_X] - \left[\frac{1}{\hbar}\left(\theta_X + \frac{1}{2}y^i \nabla_{\frac{\partial}{\partial x^i}}\theta_X - \iota_X r\right), \cdot \,\right]\\
            = &\,\,D \iota_X + \iota_X D - \left[\frac{1}{\hbar}\left(\theta_X + \frac{1}{2}y^i \nabla_{\frac{\partial}{\partial x^i}}\theta_X - \iota_X r\right), \cdot \,\right].
        \end{aligned}
    \end{equation}
    Here $\iota_X$ stands for the interior product with respect to the exterior part.
\end{lemma}

Using the preceding lemma, M\"uller-Bahns and Neumaier obtained the following useful formula.

\begin{lemma} \cite[Lemma 3.4, Lemma 3.6]{MN} \label{lemma-D-extra-terms-deformed-Cartan-formula}
    Let $X \in \Gamma^{\omega}(TM)$. We have
    \begin{equation*}
        \begin{aligned}
            D\left(\theta_X + \frac{1}{2}y^i \nabla_{\frac{\partial}{\partial x^i}}\theta_X - \iota_X r\right) = &\,- \iota_X \omega + \partial \left(\frac{1}{2}y^i \nabla_{\frac{\partial}{\partial x^i}}\theta_X\right) - \sL_X r -\iota_X R - \iota_X \Omega\\
            = &\,- \iota_X \omega - (\Gamma^X - \Gamma) - \sL_X r - \iota_X \Omega\\
            = &\,- \iota_X \omega - \eta_X.
        \end{aligned}
    \end{equation*}
\end{lemma}

\begin{remark}
    In the statement of the preceding lemmas, we translate the notations in \cite{MN} into those of our convention.
\end{remark}

For simplicity, we denote by
\[v_X = \theta_X + \frac{1}{2}y^i \nabla_{\frac{\partial}{\partial x^i}}\theta_X - \iota_X r.\]
Then, by Lemma \ref{lemma-D-extra-terms-deformed-Cartan-formula},
\begin{equation*}
    \begin{aligned}
        - \iota_X \omega - \eta_X = &\,\,Dv_X\\
        = &\,\,D(v_X - \hbar u_X + \hbar u_X)\\
        = &\,\,D(v_X - \hbar u_X) + \hbar D u_X\\
        = &\,\,D(v_X - \hbar u_X) - \eta_X.
    \end{aligned}
\end{equation*}
Hence,
\[- \iota_X \omega = D(v_X - \hbar u_X).\]
By definition, we have $\hbar u_X \in \sW$, and
\[\sigma(v_X - \hbar u_X) = 0.\]
Then, by Lemma \ref{prop-classification-derivations-star},
\[h_{- \iota_X \omega} = - v_X + \hbar u_X.\]
Therefore, for $f \in C^{\infty}(M)[[\hbar]]$, we have
\begin{equation*}
    \begin{aligned}
        D_{- \iota_X \omega}f = &\,\,\sigma\left(D_{h_{- \iota_X \omega}}\sigma^{-1}(f)\right)\\
        = &\,\,\sigma\left(\left[\frac{1}{\hbar}h_{- \iota_X \omega}, \sigma^{-1}(f)\right]\right)\\
        = &\,\,\sigma\left(\left[\frac{1}{\hbar}(- v_X + \hbar u_X), \sigma^{-1}(f)\right]\right)\\
        = &\,\,\sigma\left(\left[- \frac{1}{\hbar}v_X, \sigma^{-1}(f)\right] + [u_X, \sigma^{-1}(f)]\right).
    \end{aligned}
\end{equation*}
By Lemma \ref{lemma-deformed-Cartan-formula},
\[\left[- \frac{1}{\hbar}v_X, \sigma^{-1}(f)\right] = \sL_X \sigma^{-1}(f) - [D, \iota_X]\sigma^{-1}(f) = \sL_X \sigma^{-1}(f)\]
as $\sigma^{-1}(f) \in \sW_D = {\rm Ker} \, D \cap \sW$. Therefore,
\begin{equation*}
    \begin{aligned}
        D_{- \iota_X \omega}f = &\,\,\sigma\left(\left[- \frac{1}{\hbar}v_X, \sigma^{-1}(f)\right] + [u_X, \sigma^{-1}(f)]\right)\\
        = &\,\,\sigma\left(\sL_X \sigma^{-1}(f) + [u_X, \sigma^{-1}(f)]\right)\\
        = &\,\,\sigma\left(\widehat{X}\left(\sigma^{-1}(f)\right)\right)\\
        = &\,\,\widetilde{X}f.
    \end{aligned}
\end{equation*}
In conclusion, we have the following result.
\begin{proposition}
    For any $X \in \Gamma^{\omega}(TM)$, we have
    \[D_{- \iota_X \omega} = \widetilde{X},\]
    where $D_{- \iota_X \omega}$ is the derivation in Proposition \ref{prop-classification-derivations-star} and $\widetilde{X}$ is the derivation in Theorem \ref{thm-deformation-quantization-symplectic-vector-fields}.
\end{proposition}

\section{Cocycle phenomena of general quantizations}
In this appendix, we discuss cocycle phenomena for general deformation quantizations of symplectic vector fields based on our results in Section 4 and Proposition \ref{prop-classification-derivations-star}. We provide sufficient (and almost necessary) conditions for the existence of non-abelian $2$-cocycle of a general deformation quantization of symplectic vector fields. Moreover, we show that changing a quantization (with a non-abelian $2$-cocycle) by inner derivations yields equivalent cleft extensions.

Let $(M, \omega, \nabla)$ be a Fedosov manifold, and let $\Omega \in \hbar \wedge^2 [[\hbar]]$ be a series of closed $2$-forms. Let $X \in \Gamma^{\omega}(TM)$, and let $\widetilde{X}$ be the derivation of $(C^{\infty}(M)[[\hbar]], \star)$ in Theorem \ref{thm-deformation-quantization-symplectic-vector-fields}. Assume $\widetilde{X}' \in {\rm Der}(C^{\infty}(M)[[\hbar]], \star)$ is another derivation with $\widetilde{X}'|_{\hbar = 0} = X$. Then $\frac{\widetilde{X}' - \widetilde{X}}{\hbar}$ is also a derivation of $(C^{\infty}(M)[[\hbar]], \star)$. By Proposition \ref{prop-classification-derivations-star}, there is a uniquely formal closed $1$-form $A_X \in Z^1_{\rm dR}(M)[[\hbar]]$ such that
\[\frac{\widetilde{X}' - \widetilde{X}}{\hbar} = D_{A_X}.\]
Let $h_{A_X} \in \sW$ be given by \eqref{eq-formula-h-A} with $A = A_X$. Under the isomorphism
\[\sigma : (\sW_D, \circ) \to (C^{\infty}(M)[[\hbar]], \star),\]
the derivation $\widetilde{X}'$ corresponds to a derivation $\widehat{X}' \coloneqq \sigma^{-1} \circ \widetilde{X}' \circ \sigma$ of $\sW_D$. By the construction of $D_{A_X}$ and $D_{h_{A_X}}$, we have
\begin{equation*}
    \begin{aligned}
        \widehat{X}' = &\,\,\widehat{X} + \hbar \sigma^{-1} \circ D_{A_X} \circ \sigma\\
        = &\,\,\widehat{X} + \hbar D_{h_{A_X}}\\
        = &\,\,\widehat{X} + \hbar\left[\frac{1}{\hbar}h_{A_X}, \cdot \,\right]\\
        = &\,\,\sL_X + [u_X + h_{A_X}, \cdot \,].
    \end{aligned}
\end{equation*}
In conclusion, we have the following proposition.
\begin{proposition} \label{prop-widetilde'-X-A-X}
    Each linear map
    \[\Gamma^{\omega}(TM) \to {\rm Der}(C^{\infty}(M)[[\hbar]], \star), \quad X \mapsto \widetilde{X}'\]
    satisfying $\widetilde{X}'|_{\hbar = 0} = X$ is uniquely determined by a linear map
    \[\Gamma^{\omega}(TM) \to Z^1_{\rm dR}(M)[[\hbar]], \quad X \mapsto A_X\]
    such that
    \[\widehat{X}' \coloneqq \sigma^{-1} \circ \widetilde{X}' \circ \sigma = \sL_X + [u_X + h_{A_X}, \cdot \,] = \widehat{X} + [h_{A_X}, \cdot \,], \quad \forall X \in \Gamma^{\omega}(TM).\]
    Here $u_X$ is given by \eqref{def-u-X}, and $h_{A_X}$ is given by \eqref{eq-formula-h-A} with $A = A_X$.
\end{proposition}

Now, let
\[\Gamma^{\omega}(TM) \to {\rm Der}(C^{\infty}(M)[[\hbar]], \star), \quad X \mapsto \widetilde{X}'\]
be the linear map determined by a linear map
\[\Gamma^{\omega}(TM) \to Z^1_{\rm dR}(M)[[\hbar]], \quad X \mapsto A_X\]
as in the preceding proposition. For $X \in \Gamma^{\omega}(TM)$, denote by $\widehat{X}' = \sigma^{-1} \circ \widetilde{X}' \circ \sigma \in {\rm Der}(\sW_D)$ as above. Then, for $X, Y \in \Gamma^{\omega}(TM)$, a direct computation yields
\begin{equation*}
    \begin{aligned}
        [\widehat{X}', \widehat{Y}'] = \widehat{[X, Y]}' + &\,\big[[u_X + h_{A_X}, u_Y + h_{A_Y}] -u_{[X, Y]} - h_{A_{[X, Y]}}\\
        &\,\,+ \sL_X(u_Y + h_{A_Y}) - \sL_Y(u_X + h_{A_X}), \cdot \,\big].
    \end{aligned}
\end{equation*}
Since $D h_{A_X} = - A_X$ and $A_X$ lies in the center of $\sW$, a similar computation as in the proof of Proposition \ref{prop-D-tau-zero} gives
\begin{equation} \label{D-widehat'-XY-minus-widehat'-X-widehat'-Y}
    \begin{aligned}
        &\,\,D([u_X + h_{A_X}, u_Y + h_{A_Y}] -u_{[X, Y]} - h_{A_{[X, Y]}} + \sL_X(u_Y + h_{A_Y}) - \sL_Y(u_X + h_{A_X}))\\
        = &\,\,A_{[X, Y]} - \sL_X A_Y + \sL_Y A_X - \frac{1}{\hbar}d \iota_X \iota_Y \Omega\\
        = &\,\,A_{[X, Y]} - d\left(\iota_X A_Y - \iota_Y A_X + \frac{1}{\hbar}\iota_X \iota_Y \Omega\right).
    \end{aligned}
\end{equation}
Fedosov proved that the centralizer of $\sW_D$ in $\sW$ is $C^{\infty}(M)[[\hbar]]$ \cite[Corollary 5.5.2]{F2}. Therefore, \eqref{D-widehat'-XY-minus-widehat'-X-widehat'-Y} implies that there exists an element $\tau'(X, Y) \in \sW_D$ such that
\[[\widehat{X}', \widehat{Y}'] = \widehat{[X, Y]}' + [\tau'(X, Y), \cdot \,]\]
if and only if the formal $1$-form $A_{[X, Y]}$ is exact. In particular, we have the following proposition.

\begin{proposition} \label{prop-D-tau'-zero-sufficient-condition}
    Let
    \[\Gamma^{\omega}(TM) \to {\rm Der}(C^{\infty}(M)[[\hbar]], \star), \quad X \mapsto \widetilde{X}'\]
    be the linear map determined by a linear map
    \[\Gamma^{\omega}(TM) \to Z^1_{\rm dR}(M)[[\hbar]], \quad X \mapsto A_X\]
    as in Proposition \ref{prop-widetilde'-X-A-X}. If there is a linear map
    \[\wedge^2\Gamma^{\omega}(TM) \to C^{\infty}(M)[[\hbar]], \quad X \wedge Y \mapsto a_{X, Y}\]
    such that
    \[A_{[X, Y]} = d a_{X, Y}, \quad \forall X, Y \in \Gamma^{\omega}(TM),\]
    we then have a linear map
    \[\tau' : \wedge^2\Gamma^{\omega}(TM) \to \sW_D\]
    defined by
    \begin{equation*}
        \begin{aligned}
            \tau'(X, Y) = &\,\,[u_X + h_{A_X}, u_Y + h_{A_Y}] -u_{[X, Y]} - h_{A_{[X, Y]}} + \sL_X(u_Y + h_{A_Y}) - \sL_Y(u_X + h_{A_X})\\
            &\,+ \iota_X A_Y - \iota_Y A_X - a_{X, Y} + \frac{1}{\hbar}\iota_X \iota_Y \Omega
        \end{aligned}
    \end{equation*}
    for $X, Y \in \Gamma^{\omega}(TM)$, such that
    \[[\widehat{X}', \widehat{Y}'] = \widehat{[X, Y]}' + [\tau'(X, Y), \cdot \,], \quad \forall X, Y \in \Gamma^{\omega}(TM).\]
\end{proposition}
Now, let $X \mapsto \widetilde{X}'$ be a linear map as in the preceding proposition. Then, for $X, Y, Z \in \Gamma^{\omega}(TM)$, a similar computation as in the proof of Proposition \ref{prop-cocycle-tau} yields
\begin{equation} \label{sum-tau'-cyclic}
    \begin{aligned}
        &\,\,\tau'([X, Y], Z) + \circlearrowleft (X, Y, Z)\\
        = &\,\bigg([u_{[X, Y]} + h_{A_{[X, Y]}}, u_Z + h_{A_Z}] + \sL_{[X, Y]}(u_Z + h_{A_Z}) - \sL_X (u_{[Y, Z]} + h_{A_{[Y, Z]}})\\
        &\,\,+ \iota_{[X, Y]} A_Z - \iota_X A_{[Y, Z]} - a_{[X, Y], Z} - \frac{1}{\hbar}\Omega([X, Y], Z)\bigg) + \circlearrowleft (X, Y, Z),
    \end{aligned}
\end{equation}
and
\begin{equation*}
    \begin{aligned}
        &\,\,\widehat{X}'(\tau'(Y, Z)) + \circlearrowleft (X, Y, Z)\\
        = &\,\bigg([u_{[X, Y]} + h_{A_{[X, Y]}}, u_Z + h_{A_Z}] + \sL_{[X, Y]}(u_Z + h_{A_Z}) - \sL_X (u_{[Y, Z]} + h_{A_{[Y, Z]}})\\
        &\,\,+\sL_X \iota_Y A_Z - \sL_X \iota_Z A_Y - X a_{Y, Z} - \frac{1}{\hbar}X\Omega(Y,Z)\bigg) + \circlearrowleft (X, Y, Z).
    \end{aligned}
\end{equation*}
Since $A_Y$ is closed, we have
\[\sL_X \iota_Z A_Y = \iota_X d \iota_Z A_Y = \iota_X(d \iota_Z + \iota_Z d)A_Y = \iota_X \sL_Z A_Y.\]
Then, by Lemma \ref{lemma-L-iota},
\begin{equation*}
    \begin{aligned}
        (\sL_X \iota_Y A_Z - \sL_X \iota_Z A_Y) + \circlearrowleft (X, Y, Z)= &\,\,(\sL_X \iota_Y A_Z - \iota_X \sL_Z A_Y) + \circlearrowleft (X, Y, Z)\\
        = &\,\,(\sL_X \iota_Y - \iota_Y \sL_X)A_Z + \circlearrowleft (X, Y, Z)\\
        = &\,\,[\sL_X, \iota_Y]A_Z + \circlearrowleft (X, Y, Z)\\
        = &\,\,\iota_{[X, Y]}A_Z + \circlearrowleft (X, Y, Z).
    \end{aligned}
\end{equation*}
Therefore,
\begin{equation} \label{sum-widehat'-X-tau'-Y-Z-cyclic}
    \begin{aligned}
        &\,\,\widehat{X}'(\tau'(Y, Z)) + \circlearrowleft (X, Y, Z)\\
        = &\,\bigg([u_{[X, Y]} + h_{A_{[X, Y]}}, u_Z + h_{A_Z}] + \sL_{[X, Y]}(u_Z + h_{A_Z}) - \sL_X (u_{[Y, Z]} + h_{A_{[Y, Z]}})\\
        &\,\,+\sL_X \iota_Y A_Z - \sL_X \iota_Z A_Y - X a_{Y, Z} - \frac{1}{\hbar}X\Omega(Y,Z)\bigg) + \circlearrowleft (X, Y, Z)\\
        = &\,\bigg([u_{[X, Y]} + h_{A_{[X, Y]}}, u_Z + h_{A_Z}] + \sL_{[X, Y]}(u_Z + h_{A_Z}) - \sL_X (u_{[Y, Z]} + h_{A_{[Y, Z]}})\\
        &\,\,+ \iota_{[X, Y]} A_Z - X a_{Y, Z} - \frac{1}{\hbar}X\Omega(Y, Z)\bigg) + \circlearrowleft (X, Y, Z).
    \end{aligned}
\end{equation}
Since $A_{[Y, Z]} = d a_{Y, Z}$, we have
\begin{equation} \label{sum-iota-X-A-YZ-cyclic}
    \iota_X A_{[Y, Z]} = \iota_X d a_{Y, Z} = X a_{Y, Z}.
\end{equation}
Substitution of \eqref{Omega-XY-Z-X-Omega-YZ-cyclic-sum} and \eqref{sum-tau'-cyclic}-\eqref{sum-iota-X-A-YZ-cyclic} yields
\[\left(\tau'([X, Y], Z) - \widehat{X}'(\tau'(Y, Z))\right) + \circlearrowleft (X, Y, Z) = a_{X, [Y, Z]} + \circlearrowleft (X, Y, Z).\]
Therefore, the linear maps
\[\Gamma^{\omega}(TM) \to {\rm Der}(\sW_D), \quad X \mapsto \widehat{X}'\]
and
\[\tau' : \wedge^2 \Gamma^{\omega}(TM) \to \sW_D\]
provide a non-abelian $2$-cocycle on $\Gamma^{\omega}(TM)$ with values in $\sW_D$ if and only if
\begin{equation} \label{cocycle-a-trivial-rep}
    a_{X, [Y, Z]} + \circlearrowleft (X, Y, Z) = 0, \quad \forall X, Y, Z \in \Gamma^{\omega}(TM).
\end{equation}
Note that the condition \eqref{cocycle-a-trivial-rep} is equivalent to saying that the linear map
\[a : \wedge^2 \Gamma^{\omega}(TM) \to C^{\infty}(M)[[\hbar]], \quad X \wedge Y \mapsto a_{X, Y}\]
is a Chevalley-Eilenberg $2$-cocycle with respect to the trivial representation of $\Gamma^{\omega}(TM)$ on $C^{\infty}(M)[[\hbar]]$. In conclusion, we have the following proposition.

\begin{proposition} \label{prop-widetilde'-X-non-abelian-cocycle-conditions}
    Let
    \[\Gamma^{\omega}(TM) \to {\rm Der}(C^{\infty}(M)[[\hbar]], \star), \quad X \mapsto \widetilde{X}'\]
    be the linear map determined by a linear map
    \[\Gamma^{\omega}(TM) \to Z^1_{\rm dR}(M)[[\hbar]], \quad X \mapsto A_X\]
    as in Proposition \ref{prop-widetilde'-X-A-X}. Assume there is a linear map
    \[a : \wedge^2\Gamma^{\omega}(TM) \to C^{\infty}(M)[[\hbar]], \quad X \wedge Y \mapsto a_{X, Y}\]
    such that
    \[A_{[X, Y]} = d a_{X, Y}, \quad \forall X, Y \in \Gamma^{\omega}(TM).\]
    Then, the linear maps
    \[\Gamma^{\omega}(TM) \to {\rm Der}(\sW_D), \quad X \mapsto \widehat{X}'\]
    and
    \[\tau' : \wedge^2 \Gamma^{\omega}(TM) \to \sW_D\]
    given by Proposition \ref{prop-D-tau'-zero-sufficient-condition} provide a non-abelian $2$-cocycle if and only if the linear map $a : \wedge^2\Gamma^{\omega}(TM) \to C^{\infty}(M)[[\hbar]]$ is a Chevalley-Eilenberg $2$-cocycle with respect to the trivial representation of $\Gamma^{\omega}(TM)$ on $C^{\infty}(M)[[\hbar]]$, i.e.,
    \[a_{X, [Y, Z]} + \circlearrowleft (X, Y, Z) = 0, \quad \forall X, Y, Z \in \Gamma^{\omega}(TM).\]
\end{proposition}

If we have a quantization map
\[\Gamma^{\omega}(TM) \to {\rm Der}(C^{\infty}(M)[[\hbar]], \star), \quad X \mapsto \widetilde{X}'\]
satisfying the conditions in Proposition \ref{prop-widetilde'-X-non-abelian-cocycle-conditions}, the results of Section 4 indicate that this quantization can also provide certain (completed) cleft extensions and cocycle crossed products. In the following, we show that adding inner derivations of $(C^{\infty}(M)[[\hbar]], \star)$ yields equivalent cleft extensions.

Let
\[b : \Gamma^{\omega}(TM) \to C^{\infty}(M)[[\hbar]], \quad X \mapsto b_X\]
be a linear map. We then have a linear map
\[A : \Gamma^{\omega}(TM) \to Z^1_{\rm dR}(M)[[\hbar]], \quad X \mapsto A_X\]
defined by
\[A_X = d b_X, \quad \forall X \in \Gamma^{\omega}(TM).\]
Let
\[\Gamma^{\omega}(TM) \to {\rm Der}(C^{\infty}(M)[[\hbar]], \star), \quad X \mapsto \widetilde{X}'\]
be the linear map given by the preceding linear map $A$ as in Proposition \ref{prop-widetilde'-X-A-X}. Then, by Proposition \ref{prop-classification-derivations-star},
\[\widetilde{X}' = \widetilde{X} + [b_X, \cdot \,], \quad \forall X \in \Gamma^{\omega}(TM).\]
Since
\[A_{[X, Y]} = d b_{[X, Y]}, \quad \forall X, Y \in \Gamma^{\omega}(TM),\]
Proposition \ref{prop-D-tau'-zero-sufficient-condition} yields a linear map
\[\tau' : \wedge^2 \Gamma^{\omega}(TM) \to \sW_D\]
defined by
\begin{equation*}
    \begin{aligned}
        \tau'(X, Y) = &\,\,[u_X + h_{A_X}, u_Y + h_{A_Y}] -u_{[X, Y]} - h_{A_{[X, Y]}} + \sL_X(u_Y + h_{A_Y}) - \sL_Y(u_X + h_{A_X})\\
        &\,+ \iota_X A_Y - \iota_Y A_X - b_{[X, Y]} + \frac{1}{\hbar}\iota_X \iota_Y \Omega\\
        = &\,\,[u_X + h_{A_X}, u_Y + h_{A_Y}] -u_{[X, Y]} - h_{A_{[X, Y]}} + \sL_X(u_Y + h_{A_Y}) - \sL_Y(u_X + h_{A_X})\\
        &\,+ X b_Y - Y b_X -b_{[X, Y]} + \frac{1}{\hbar}\iota_X \iota_Y \Omega
    \end{aligned}
\end{equation*}
for $X, Y \in \Gamma^{\omega}(TM)$, such that
\[[\widehat{X}', \widehat{Y}'] = \widehat{[X, Y]}' + [\tau'(X, Y), \cdot \,], \quad \forall X, Y \in \Gamma^{\omega}(TM).\]
Since the map
\[b : \Gamma^{\omega}(TM) \to C^{\infty}(M)[[\hbar]]\]
is linear, the Jacobi identity gives
\[b_{[X, [Y, Z]]} + \circlearrowleft (X, Y, Z) = 0, \quad \forall X, Y, Z \in \Gamma^{\omega}(TM).\]
Therefore, Proposition \ref{prop-widetilde'-X-non-abelian-cocycle-conditions} indicates that the linear maps
\[\Gamma^{\omega}(TM) \to {\rm Der}(\sW_D), \quad X \mapsto \widehat{X}'\]
and
\[\tau' : \wedge^2 \Gamma^{\omega}(TM) \to \sW_D\]
provide a non-abelian $2$-cocycle. Let $\sW_D \times_{\tau} \Gamma^{\omega}(TM)$ be the Lie algebra defined by the non-abelian $2$-cocycle in Proposition \ref{prop-cocycle-tau} with the Lie bracket defined as in \eqref{Lie-bracket-cocycle}. Let $\sW_D \times_{\tau'} \Gamma^{\omega}(TM)$ be the Lie algebra defined by the preceding non-abelian $2$-cocycle in the same way. Since
\[D(h_{A_X} + b_X) = Dh_{A_X} + d b_X = - A_X + A_X = 0, \quad \forall X \in \Gamma^{\omega}(TM),\]
we have the following linear map
\[\gamma : \sW_D \times_{\tau} \Gamma^{\omega}(TM) \to \sW_D \times_{\tau'} \Gamma^{\omega}(TM)\]
defined by
\begin{equation} \label{def-gamma-isom}
    \gamma(\xi, X) = (\xi - (h_{A_X} + b_X), X), \quad \forall \xi \in \sW_D, \,\, \forall X \in \Gamma^{\omega}(TM).
\end{equation}
It is clear that $\gamma$ is a linear isomorphism with the inverse map
\[\gamma^{-1}(\xi, X) = (\xi + (h_{A_X} + b_X), X), \quad \forall \xi \in \sW_D, \,\, \forall X \in \Gamma^{\omega}(TM).\]
In fact, $\gamma$ is a Lie algebra isomorphism, as the following proposition shows.

\begin{proposition}
    The linear isomorphism $\gamma$ defined by \eqref{def-gamma-isom} is a Lie algebra isomorphism. In particular, the cleft extensions $\sW_D \lcross_{\tau} \sU(\Gamma^{\omega}(TM))$ and $\sW_D \lcross_{\tau'} \sU(\Gamma^{\omega}(TM))$ are equivalent.
\end{proposition}
\begin{proof}
    We denote by $[\cdot\,, \cdot]$ and $[\cdot \, , \cdot]'$ the Lie brackets of $\sW_D \times_{\tau} \Gamma^{\omega}(TM)$ and $\sW_D \times_{\tau'} \Gamma^{\omega}(TM)$, respectively. Then, for $\xi, \zeta \in \sW_D$, $X, Y \in \Gamma^{\omega}(TM)$, we have
    \[[(\xi, X), (\zeta, Y)] = ([\xi, \zeta] + \widehat{X}\zeta - \widehat{Y}\xi + \tau(X, Y), [X, Y]).\]
    Thus, by definition,
    \begin{equation} \label{gamma-bracket-xi-X-zeta-Y}
        \gamma([(\xi, X), (\zeta, Y)]) = ([\xi, \zeta] + \widehat{X}\zeta - \widehat{Y}\xi + \tau(X, Y) - h_{A_{[X, Y]}} - b_{[X, Y]}, [X, Y]).
    \end{equation}
    Since $b_X, b_Y$ lie in the center of $\sW$, we have
    \begin{equation} \label{bracket-gamma-xi-X-gamma-zeta-Y}
        \begin{aligned}
            &\,\,[\gamma(\xi, X), \gamma(\zeta, Y)]'\\
            = &\,\,[(\xi - (h_{A_X} + b_X), X), (\zeta - (h_{A_Y} + b_Y), Y)]'\\
            =&\,\,([\xi - h_{A_X}, \zeta - h_{A_Y}] + \widehat{X}'(\zeta - h_{A_Y} - b_Y) - \widehat{Y}'(\xi - h_{A_X} - b_X) + \tau'(X, Y), [X, Y]).
        \end{aligned}
    \end{equation}
    Since
    \[\widehat{X}' = \widehat{X} + [h_{A_X}, \cdot\,], \quad \forall X \in \Gamma^{\omega}(TM),\]
    we have
    \begin{equation} \label{bracket-xi-h-A-X-zeta-h-A-Y}
        \begin{aligned}
            &\,\,[\xi - h_{A_X}, \zeta - h_{A_Y}] + \widehat{X}'(\zeta - h_{A_Y} - b_Y) - \widehat{Y}'(\xi - h_{A_X} - b_X)\\
            = &\,\,[\xi, \zeta] - [h_{A_X}, \zeta] + [h_{A_Y}, \xi] + [h_{A_X}, h_{A_Y}]\\
            &\,+ \widehat{X}(\zeta - h_{A_Y} - b_Y) + [h_{A_X}, \zeta - h_{A_Y}] - \widehat{Y}(\xi - h_{A_X} - b_X) - [h_{A_Y}, \xi - h_{A_X}]\\
            = &\,\,[\xi, \zeta] + \widehat{X}(\zeta - h_{A_Y} - b_Y) - \widehat{Y}(\xi - h_{A_X} - b_X) - [h_{A_X}, h_{A_Y}]\\
            = &\,\,[\xi, \zeta] + \widehat{X}\zeta - \widehat{Y}\xi - \widehat{X}h_{A_Y} + \widehat{Y}h_{A_X} - X b_Y + Y b_X - [h_{A_X}, h_{A_Y}].
        \end{aligned}
    \end{equation}
    Moreover, by definition,
    \begin{equation} \label{difference-tau-tau'-exact-forms}
        \begin{aligned}
            \tau'(X, Y) = &\,\,[u_X + h_{A_X}, u_Y + h_{A_Y}] -u_{[X, Y]} - h_{A_{[X, Y]}} + \sL_X(u_Y + h_{A_Y}) - \sL_Y(u_X + h_{A_X})\\
            &\,+ \iota_X A_Y - \iota_Y A_X - b_{[X, Y]} + \frac{1}{\hbar}\iota_X \iota_Y \Omega\\
            = &\,\,\tau(X, Y) + [h_{A_X}, h_{A_Y}] - h_{A_{[X, Y]}} + \widehat{X}h_{A_Y} - \widehat{Y}h_{A_X}\\
            &\,+ \iota_X A_Y - \iota_Y A_X - b_{[X, Y]}\\
            = &\,\,\tau(X, Y) + [h_{A_X}, h_{A_Y}] - h_{A_{[X, Y]}} + \widehat{X}h_{A_Y} - \widehat{Y}h_{A_X}\\
            &\,+ \iota_X d b_Y - \iota_Y d b_X - b_{[X, Y]}\\
            = &\,\,\tau(X, Y) + [h_{A_X}, h_{A_Y}] - h_{A_{[X, Y]}} + \widehat{X}h_{A_Y} - \widehat{Y}h_{A_X}\\
            &\,+ X b_Y - Y b_X - b_{[X, Y]}.
        \end{aligned}
    \end{equation}
    Substitution of \eqref{gamma-bracket-xi-X-zeta-Y}-\eqref{difference-tau-tau'-exact-forms} yields
    \[\gamma([(\xi, X), (\zeta, Y)]) = [\gamma(\xi, X), \gamma(\zeta, Y)]'.\]
    Therefore, $\gamma$ is a Lie algebra isomorphism. Since
    \[\gamma|_{\sW_D} = {\rm Id}_{\sW_D},\]
    $\gamma$ extends to an algebra isomorphism
    \[\tilde{\gamma} : \sW_D \lcross_{\tau} \sU(\Gamma^{\omega}(TM)) \xrightarrow{\cong} \sW_D \lcross_{\tau'} \sU(\Gamma^{\omega}(TM))\]
    which restricts to the identity of the subalgbra $\sW_D$. Moreover, the explicit proof in \cite{Mo} indicates that $\tilde{\gamma}$ is also compatible with the $\sU(\Gamma^{\omega}(TM))$-comodule structures. Therefore, $\sW_D \lcross_{\tau} \sU(\Gamma^{\omega}(TM))$ and $\sW_D \lcross_{\tau'} \sU(\Gamma^{\omega}(TM))$ are equivalent cleft extensions of $\sW_D$.
\end{proof}

\begin{remark}
    All the preceding discussions are valid if we replace the initial quantization $X \mapsto \widetilde{X}$ by an arbitrary quantization $X \mapsto \widetilde{X}'$ which also satisfies the non-abelian $2$-cocycle conditions as in Definition \ref{def-non-abelian-2-cocycle-Lie}. In particular, changing a quantization (with a non-abelian $2$-cocycle) by inner derivations yields equivalent cleft extensions.
\end{remark}

\bigskip
Haoyuan Gao (hygao@math.ecnu.edu.cn)

School of Mathematical Sciences, Key Laboratory of MEA (Ministry of Education) $\&$ Shanghai Key Laboratory of PMMP, East China Normal University, Shanghai 200241, China


\end{document}